\documentclass[12pt]{amsart}
\usepackage{amsmath,amsfonts,amssymb,color}
\usepackage{a4wide,amsthm}
\usepackage[english]{babel}
\usepackage{graphicx}
\usepackage{hyperref,pdfsync}
\usepackage{tikz}

\newcommand{\Kappa}{\mathcal K}
\renewcommand{\leq}{\leqslant}
\renewcommand{\geq}{\geqslant}

\newtheorem{theorem}{Theorem}
\newtheorem{corollary}[theorem]{Corollary}
\newtheorem{proposition}[theorem]{Proposition}
\newtheorem{lemma}[theorem]{Lemma}
\newtheorem{definition}[theorem]{Definition}
\newcommand{\card}{{\mathbf{\tiny \#}}}

\begin{document}

\title[Homogenization of the Stokes problem]{On the homogenization of the Stokes problem in a perforated domain}
\author{ M. Hillairet}
\address{Institut Montpelli\'erain Alexander Grothendieck, Universit\'e de Montpellier, Place Eug\`ene Bataillon, 34095 Montpellier Cedex 5 France}
\email{matthieu.hillairet@umontpellier.fr}
\date{\today}

\maketitle

\begin{abstract}
We consider the Stokes equations on a bounded perforated domain completed with non-zero
constant boundary conditions on the holes. We investigate configurations for which the holes are identical spheres and their number $N$ goes to infinity while 
their radius $a^N$ tends to zero.  Under the assumption that $a^N$ scales like $a/N$ and that there is no concentration in the distribution of holes, we prove that the solution is well approximated asymptotically by solving
a Stokes-Brinkman problem.
\end{abstract}

\medskip

\section{Introduction}
Let $\Omega$ be a smooth bounded domain in $\mathbb R^3.$ Given $N \in \mathbb N,$ let
$a^N > 0,$ $(h_1^N,\ldots,h_{N}^N)$   in $\Omega,$ such that the $B_i^{N} = B(h_{i}^N,a^N)$  satisfy 
\begin{equation} \label{eq_ass0} \tag{A0}
B_i^N  \Subset \Omega\,,   \qquad \overline{B_i^N} \cap \overline{B_j^N} = \emptyset\,, \qquad \text{ for  $i \neq j$  in $\{1, \ldots, N\}\,,$}
\end{equation}
and consider a $N$-uplet $(v_i^{N})_{i=1,\ldots,N} \in (\mathbb R^3)^{N}.$ It is classical that there exists a unique solution to  
\begin{equation} \label{eq_stokesN}
\left\{
\begin{array}{rcl}
- \Delta u + \nabla p &=& 0\,, \, \\
{\rm div} \, u &= & 0 \,,
\end{array}
\right.
\quad \text{ on $\mathcal F^{N} := \Omega \setminus \bigcup_{i=1}^N \overline{B_i^N}$}\,,
\end{equation}
completed with boundary conditions 
\begin{equation} \label{cab_stokesN}
\left\{
\begin{array}{rcll}
u &=& v_i^N  \,, &  \text{on $\partial  B_i^{N}$} \,, \quad \forall \, i =1,\ldots,N\,, \\
u &=& 0 \,, & \text{on $\partial \Omega$}\,.
\end{array}
\right.
\end{equation}
We are interested here in the behavior of this solution when $N$ goes to infinity and the asymptotics of the data 
$(h_i^N,v_i^N)_{i=1,\ldots,N}$ are given. 

\medskip

The closely related problem of periodic homogenization of the Stokes equations in a bounded domain perforated by tiny holes
is considered in  \cite{Allaire}. It is proven therein that there exists a critical value of the ratio between the size of the holes and their minimal distance for which 
the homogenized problem is a Stokes-Brinkman problem. 
If the holes are "denser" the homogenized problem is of Darcy type while if the holes are "more dilute" one obtains again a Stokes problem. 
This former result is an adaptation to the Stokes equations of a previous analysis on the Laplace equation in \cite{CioranescuMurat}. 
We  refer the reader to \cite{BonLacMas15,FeireislYong,LacMas16} for a review of equivalent results for other fluid models. 

In \cite{Allaire}, the Stokes equations are completed with vanishing boundary conditions while a volumic source term is added in the bulk. 
The very problem that we consider herein \eqref{eq_stokesN}-\eqref{cab_stokesN}, with non-zero constant boundary  conditions,  is introduced in \cite{DGR} 
 for the modeling of a thin spray in a highly viscous fluid. In this case, the holes represent droplets of another phase called "dispersed phase". 
 This phase can be made of another fluid or  small rigid spheres. The Stokes equations should then be completed with evolution equations for this dispersed phase yielding a time-evolution 
problem with moving holes.  With this application in mind, computing the asymptotics of the stationary Stokes problem \eqref{eq_stokesN}-\eqref{cab_stokesN}
is a tool for understanding the instantaneous response of the dispersed phase to the drag forces exerted by 
the  flow on the droplets/spheres. We refer the reader to \cite{DGR,JP} for more details on the modeling. 
In \cite{DGR}, the authors adapt the result of \cite{Allaire} on the derivation of the Stokes-Brinkman system. 
%A comparable analysis with another purpose is provided in \cite{OJ}. 
We emphasize that there is a significant new difficulty in introducing non-vanishing boundary conditions. Indeed, the boundary conditions
on the holes may be highly oscillating (when jumping from one hole to another). Hence, if one was trying to compute the homogenized system for \eqref{eq_stokesN}-\eqref{cab_stokesN} by lifting the boundary conditions, it would introduce
a highly oscillating source term in the Stokes equations that is out of the scope of the analysis in \cite{Allaire}.

The result in \cite{DGR} is obtained under the assumption that $a^N=1/N$ and that the distance between two centers $h_i^N$ and $h_j^N$ is larger than $2/N^{1/3}.$ 
The first assumption is natural since, as explained in this reference, it implies that the collective repulsion force applied by the holes on the fluid is of order one.
On the other hand, the second assumption is quite restrictive. Indeed, first, if one were choosing the centers $(h_i^N)_{i=1,\ldots,N}$ randomly as in \cite{Rubinstein}, the set containing such configurations would be
asymptotically negligible. The second limitation appears in the classical case where the holes are rigid particles moving according to Newton laws.
Indeed, in this time-dependant case, even if the particles are distibuted initially so that their centers are sufficiently distant, it is likely that this condition on the minimum distance is bound to be broken instantaneously,  except if the initial velocities of the particles are correlated with the initial positions of their centers.  Our main motivation in this paper is to provide another approach that may help to overcome these difficulties. 

\medskip

In order to consider the limit $N\to \infty$, we make now  precise the different assumptions on the data of our Stokes problem
\eqref{eq_stokesN}-\eqref{cab_stokesN}. This includes: 
\begin{itemize}
\item the positions of the centers  $(h_i^{N})_{i=1,\ldots,N},$\\[-8pt]
\item the velocities prescribed on the holes $(v_i^N)_{i=1,\ldots,N}.$
\end{itemize}
First, similarly to \cite{DGR}, we consider data so that:
\begin{equation}
\dfrac{1}{N} \sum_{i=1}^N |v^N_i|^2  \text{ is uniformly bounded}\,. \label{eq_ass1} \tag{A1}\\
\end{equation}
We name such configurations "finite-energy." Indeed, the "energy" associated with solving the Stokes problem \eqref{eq_stokesN}  is what is also called the "dissipation" in the time-evolution case: 
$$
\int_{\mathcal F^N} |\nabla u|^2.
$$
We shall show that the assumption \eqref{eq_ass1} (with \eqref{eq_ass_dmin} below) entails that this energy is bounded independently of $N.$

\medskip

Second, we introduce the empirical measures:
$$
S_N = \dfrac{1}{N} \sum_{i=1}^{N} \delta_{h^N_i,v^N_i} \in \mathbb P(\mathbb R^3 \times \mathbb R^3),
$$
where $\delta_{h,v}$ denotes the Dirac mass centered in $(h,v) \in \mathbb R^3 \times \mathbb R^3,$ and we assume:
\begin{align}
\int_{\mathbb R^3} S_N({\rm d}v ) \rightharpoonup \rho(x){\rm d}x  & \text{ weakly in the sense of measures on $\mathbb R^3\,,$} \label{eq_ass4} \tag{A2} \\
\int_{\mathbb R^3} v S_N({\rm d}v) \,  \rightharpoonup j(x){\rm d}x  & \text{ weakly in the sense of (vectorial-)measures on $\mathbb R^3\,.$} \label{eq_ass5}  \tag{A3}
\end{align}
We recall that, by assumption \eqref{eq_ass0}, the measure $S_N$ is supported in $\Omega \times \mathbb R^3$ so that, in the weak limit, $\rho \geq 0$
and $\rho$ and $j$ have support included in $\Omega.$ 

\medskip

As in \cite{Allaire,DGR}, we also make precise the size of the holes and the dilution regime that we consider. To quantify this, 
we introduce:
\begin{align*}
& d_{min}^N = \min_{i=1,...,N} \left\{ \text{dist}(h^N_i, \partial \Omega) ,  \min_{j\neq i} |h_i^N- h_j^N| \right\}\,.%\label{eq_dmin}
\end{align*}
First, we assume that the radii $a^N$ scale like $a/N$ and that the holes do not see each other at their own scale:
\begin{equation}\tag{A4}
 \lim_{N\to \infty} Na^N = a >0 ,  \qquad \quad \lim_{N \to \infty} {N d_{min}^N}  = +\infty\,.   \label{eq_ass_dmin} 
\end{equation}
Second, we assume that there exists a sequence $(\lambda^N)_{N\in \mathbb N} \in (0,\infty)^{\mathbb N}$ for which:
\begin{align}
& \sup_{N \in \mathbb N} \; \dfrac{1}{N |\lambda^N|^{3}}  \; \sup_{x \in \Omega} \card \left\{ i \in \{1,\ldots,N \} \text{ s.t. } h_i^N \in \overline{B(x,\lambda^N)}\right\} < \infty \label{eq_ass_concentration} \tag{A5}
\end{align}
and that this sequence satisfies the compatibility condition:
\begin{equation} \tag{A6} \label{eq_ass_lambdaN} 
\sup_{N \in \mathbb N} \lambda^N |d_{min}^N|^{-1/3} < \infty ,  \qquad
 \qquad  \lim_{N \to \infty} N^{1/6}\lambda^N = 0.
\end{equation}
We comment on these assumptions and their optimality later on.

\medskip

For $N$ sufficiently large,  \eqref{eq_ass_dmin} implies that the $(B_i^N)_{i=1,\ldots,N}$ are disjoint
and do not intersect $\partial \Omega.$ Hence, for $N$ large enough, assumption \eqref{eq_ass0} only fixes that the holes are inside $\Omega$. 
Again, there exists then a unique pair $(u^N,p^N) \in H^1(\mathcal F^N) \times L^2(\mathcal F^N)$ solution to \eqref{eq_stokesN}-\eqref{cab_stokesN}  (see next section for more details). 
 The pressure is unique up to an additive constant that we may fix
by requiring that $p^N$ has mean $0$. It can be seen as the Lagrange multiplier of the divergence-free condition in \eqref{eq_stokesN}.
Hence, we focus on the convergence of the sequence $(u^N)_{N \in \mathbb N}$ and will not go into details on what happens to the pressure (in contrast
with \cite{Allaire}). The $u^N$ are defined on different domains. In order to compute a limit for this sequence
of vector-fields, we unify their domain of definition by extending $u^N$ with the values $v_i^N$ on $B_i^N$
for any $i=1,\ldots,N.$ We still denote $u^N$ the extension for simplicity. This is now a sequence in $H^1_0(\Omega).$
Our main result reads:
\begin{theorem} \label{thm_main}
Let $(v_i^N,h_i^N)_{i=1,\ldots,N}$ be a sequence of data satisfying  \eqref{eq_ass0} for arbitrary $N \in \mathbb N$ and  \eqref{eq_ass1}--\eqref{eq_ass5} with 
$j \in L^2({\Omega}),$  $\rho \in L^{\infty}({\Omega})\,.$
Assume furthermore that \eqref{eq_ass_dmin}--\eqref{eq_ass_lambdaN} hold true.
Then, the associated sequence of extended velocity-fields $(u^N)_{N\in \mathbb N}$ converges
in $H^1_{0}(\Omega)-w$ to the unique velocity-field $\bar{u} \in H^1(\Omega)$ such that there exists a pressure  $\bar{p} \in L^2(\Omega)$ 
for which $(\bar u,\bar p)$ solves:
\begin{equation} \label{eq_SB}
\left\{
\begin{array}{rcl}
- \Delta \bar{u} + \nabla \bar{p} &=& 6 \pi a(j - \rho \bar{u}) \,, \\
{\rm div} \, \bar{u} &= & 0 \,,
\end{array}
\right.
\quad \text{ on $\Omega,$}
\end{equation}
completed with boundary conditions 
\begin{equation} \label{cab_SB}
\bar{u}= 0 \,,  \quad  \text{on $\partial \Omega$}\,.
\end{equation}
\end{theorem} 

Concerning the assumptions of our theorem, we mention that,  with \eqref{eq_ass1} and \eqref{eq_ass_concentration}-\eqref{eq_ass_lambdaN}, we may extract a subsequence such that  the  first momentums of $S_N$ in $v$ converge to some $(\rho,j) \in L^{\infty}(\Omega) \times L^2(\Omega)$.  Hence, assumptions \eqref{eq_ass4} and \eqref{eq_ass5} only fix that the whole sequence converges to the same density $\rho$ and momentum
distribution $j$. We also note that we do not include a source term $f \in L^2(\Omega)$ (independant of $N$) in \eqref{eq_stokesN} even if our result extends in a straightfoward way 
to this case (due to the linearity of the Stokes equations).
Conversely,  if the empirical measures $S_N$ converge in the sense of \eqref{eq_ass4} to a bounded density $\rho \in L^{\infty}(\Omega),$ standard measure-theory arguments
show that there exists a sequence $(\lambda^N)_{N\in \mathbb N}$ so that
\eqref{eq_ass_concentration} holds true. We emphasize that, if this property is satisfied for some sequence $(\lambda^N)_{N\in \mathbb N}$, the same property is also satisfied by any  sequence $(\tilde{\lambda}^N)_{N\in \mathbb N}$ such that $\tilde{\lambda}^N \geq \lambda^N$ uniformly.
Then \eqref{eq_ass_lambdaN} might be interpreted as a compatibility condition between the minimal distances $(d_{min}^N)_{N\in \mathbb N}$ and the largest possible sequence $(\lambda^N)_{N\in \mathbb N}$.  

Another approach on the homogenization of the Stokes problem in a perforated domain relies on the notion of screening length (see \cite{HofVelpp} for instance).
We do not state our assumptions in these terms herein. However, a comparable set of assumptions to  \eqref{eq_ass_dmin}-\eqref{eq_ass_concentration}-\eqref{eq_ass_lambdaN} is introduced in \cite{NietVel06} to study the Laplace equations in perforated domains. In this reference, the relations between this set of assumptions and the screening property is discussed.

\medskip

The framework we identify with \eqref{eq_ass_dmin}-\eqref{eq_ass_lambdaN} represents  a non-trivial extension of previous computations in \cite{Allaire} and \cite{DGR}. First, if the distribution of holes is periodic as in \cite{Allaire}, we remark that, $d^N_{min} \lesssim 1/N^{1/3} \lesssim \lambda^N$  
and
$$
\sup_{x \in \Omega} \card \left\{ i \in \{1,\ldots,N \} \text{ s.t. } h_i^N \in \overline{B(x,\lambda^N)}\right\} = \dfrac{|\lambda^N|^3}{|d^N_{min}|^3}
$$
Consequently, the assumption \eqref{eq_ass_concentration} is satisfied if and only if $d_{min}^N \sim 1/N^{1/3}.$  If the radius of the holes is $a/N,$ we recover the critical value for the cell dimensions that is found in \cite{Allaire}.
Furthermore, in the periodic case, the density and flux $\rho$ and $j$ are constant so that the Stokes-Brinkman system we derive with {\bf Theorem \ref{thm_main}} is the same as the one of \cite{Allaire}. We mention that the two other non-critical regimes of \cite{Allaire}  are incompatible with our set of assumptions. The non-periodic configurations of \cite{DGR} are also included in our set of assumptions in the case $d_{min}^N$ behaves like $1/N^{1/3}.$  We recover again the same Stokes-Brinkman system  as in this former reference. 
But, the two assumptions \eqref{eq_ass_concentration}--\eqref{eq_ass_lambdaN} include also a lot more configurations in order that one may tackle the case of random configurations (see \cite{CarHill}). 

Another novelty of the paper stems from the method of proof. We apply herein arguments that are not highly sensitive to the explicit value of solutions to the Stokes problem. Our proof relies on the weak-formulation of \eqref{eq_stokesN} and the two main ingredients are the decrease of stokeslets (see \eqref{eq_Stokesexpansion}) and conservation arguments (see next subsection). We expect that our method can be extended to the full nonlinear Navier-Stokes equations -- as in \cite{DGR}. Also, we obtain an equivalent result for holes with arbitrary shapes and boundary conditions including rotation-velocities  on the holes (not only translation, see \cite{MSH} for more details).
We think that the content of this paper and of \cite{MSH} shall help tackling  the time-evolution problem with particles moving according to Newton laws. 
A homogenized system for such a time-dependent problem is computed in \cite{Hofpp} under the assumption that the particles have no inertia. We emphasize here that, in case of inertialess particles, the Newton laws degenerate into a system of nonlinear equations correlating the positions of particles and their velocities. It is then possible to propagate in time the regime
of \cite{DGR,HofVelpp} in which the minimal distance between particles is larger than $1/N^{1/3}.$ This is the regime under consideration in \cite{Hofpp} 
extending \cite{OJ} where the authors had proven that the regime where the minimal distance is much larger than $1/N^{1/3}$ is preserved locally in time. The case of particles with inertia is still broadly open.

\subsection{Outline of the proof}

Our proof is based on a classical compactness method.
First, we prove that the sequence $(u^N)_{N \in \mathbb N}$ is bounded in $H^{1}_0(\Omega).$ 
This part is obtained by applying a variational characterization of solutions to Stokes problems and relies only upon  \eqref{eq_ass1} and \eqref{eq_ass_dmin}.  We may then extract a subsequence (that we do not relabel) converging to some $\bar{u}$ in $H^1_0(\Omega)$
(and strongly in any $L^q(\Omega)$ for $q \in [1,6[$). In order to identify a system satisfied by $\bar{u}$ all that remains is devoted to the proof that:
$$
I_w := \int_{\Omega} \nabla \bar{u} : \nabla w \,, 
$$
satisfies:
$$
I_w = 6\pi a \int_{\Omega} (j(x) - \rho(x) {\bar u}(x)) \cdot w(x) {\rm d}x\,,
$$ 
for arbitrary divergence-free $w \in C^{\infty}_c(\Omega).$
So, we fix a divergence-free $w \in C^{\infty}_c(\Omega)$ and we note that, by construction, we have 
$$
I_w = \lim_{N \to \infty} I_w^N \qquad \text{ with } \quad I_w^N = \int_{\Omega} \nabla u^N : \nabla w \,, \quad \forall \, N \in \mathbb N.
$$
We compute then $I_w^N$ by applying that $u^N$ is a solution to the Stokes problem \eqref{eq_stokesN}-\eqref{cab_stokesN}.
As the support of all the integrals $I_w^N$ is $\Omega$ and the support of $w$ is not adapted to the Stokes problem \eqref{eq_stokesN}-\eqref{cab_stokesN},
this requires special care. 

\medskip

Following the line of \cite{DGR}, we compute the integral $I_w^N$ by dividing it
into the sum of contributions due to cells around the particles. However, the minimal distance 
between particles that we allow is too small in order that each cell contains only one particle (as in \cite{DGR}). 
So, we use as cells a covering  $(T^N_{\kappa})_{\kappa \in \Kappa^N}$ of $\text{Supp}(w)$ with cubes of width $\lambda^N$ and we split 
$$
I_w^N = \sum_{\kappa \in \Kappa^N} \int_{T_{\kappa}^N} \nabla u^N : \nabla w.  
$$
This leads us to sum the contribution of the holes by packs (corresponding to  holes belonging to the same cell of the partition).  Precisely, given $N$ and $\kappa,$ we apply that there are not too many holes in $T^N_{\kappa}$ because of assumption \eqref{eq_ass_concentration}.
Under the restriction \eqref{eq_ass_lambdaN}, we are able to replace $w$  by
\begin{equation} \label{eq_wreplace}
\sum_{i \in \mathcal I^N_{\kappa}} U^{a^N}[w(h_i^N)](x-h_i^N)\,,
\end{equation}
in the integral on  $T_{\kappa}^N.$  We denote here 
\begin{itemize}
\item $\mathcal I^N_{\kappa}$  the subset of indices $i \in \{1,\ldots,N\}$ for which $h_i^{N} \in T^{N}_{\kappa},$ \\[-8pt]
\item $(U^{a^N}[v](y),P^{a^N}[v](y))$ the  solution to the Stokes problem outside $B(0,a^N)$  with boundary condition $U[v](y) = v$ on $\partial B(0,a^N)$
and vanishing condition at infinity. 
\end{itemize}
We obtain that
$$
\int_{T_{\kappa}^N} \nabla u^N : \nabla w \sim \sum_{i \in \mathcal I^N_{\kappa}} \int_{T_{\kappa}^N} \nabla u^N : \nabla [U^{a^N}[w(h_i^N)]](x-h_i^N)\,.
$$
Then, we observe that the pair
$$
\left(  U^{a^N}[w(h_i^N)](x-h_i^N),  P^{a^N}[w(h_i^N)](x-h_i^N) \right)
$$
is a solution to the Stokes problem outside $B_i^N.$ Hence, we apply that $u^N$ is divergence-free, introduce the pressure and integrate by parts to obtain that:
\begin{multline*}
\int_{T_{\kappa}^N} \nabla u^N : \nabla w  \sim  \sum_{i \in \mathcal I^N_{\kappa}} \int_{\partial T_{\kappa}^N}   (\partial_n U^{a^N}[w(h_i^N)]  - P^{a^N}[w(h_i^N)]n) \cdot u^N{\rm d}\sigma   \\ 
 \; - \int_{\partial B_i^N} (\partial_n U^{a^N}[w(h_i^N)] - P^{a^N}[w(h_i^N)]n) \cdot v_i^N {\rm d}\sigma \,.
\end{multline*}
We skip for conciseness that $(U^{a^N},P^{a^N})$ depends on $(x-h_i^N)$ in these last identities.
It is classical by the Stokes law that:
$$
\int_{\partial B_i^N} (\partial_n U^{a^N}[w(h_i^N)] - P^{a^N}[w(h_i^N)]n )  {\rm d}\sigma = - 6 \pi a^N w(h_i^N)  \,, 
$$
and, by interpreting the Stokes system as the conservation of normal stress,  that:
$$
\int_{\partial T_{\kappa}^N} (\partial_n U^{a^N}[w(h_i^N)] - P^{a^N}[w(h_i^N)]n) {\rm d}\sigma = - {6 \pi}{a^N} w(h_i^N) \,.
$$
To take advantage of this last identity, we use that the size of $T_{\kappa}^N$ decreases to $0$ and we replace $u^N$ by some mean value $\bar{u}^N_{\kappa}$ in the integral on $\partial T_{\kappa}^N.$
Say for simplicity that:  
\begin{equation} \label{eq_ukappa1}
\bar{u}^N_{\kappa} = \dfrac{1}{|T_{\kappa}^N|} \int_{T_{{\kappa}}^N} u^N(x){\rm d}x\,,
\end{equation}
and assume that replacing $u^N$ by $\bar{u}_{\kappa}^N$ induces a small  error in the boundary integral. We obtain then that:
 $$
 \int_{T_{\kappa}^N} \nabla u^N : \nabla w  \sim  \sum_{i\in \mathcal I_{\kappa}^N} {6\pi}a^{N} w(h_i^N) \cdot v_i^N -  \sum_{i\in \mathcal I_{\kappa}^N} {6\pi}a^{N} w(h_i^N) \cdot \bar{u}_{\kappa}^N .
 $$
Summing over $\kappa$ yields:
$$
I^N_{w}  \sim   6\pi Na^N \left[ \dfrac{1}{N}\sum_{i=1}^N w(h_i^N) \cdot v_i^N - \dfrac{1}{N} \sum_{\kappa \in \Kappa^N} \left[ \sum_{i\in \mathcal I_{\kappa}^N}  w(h_i^N) \right] \cdot \bar{u}_{\kappa}^N\right].
$$
By assumptions \eqref{eq_ass_dmin} and \eqref{eq_ass5} we have respectively that $Na^N$ converges to $a$
and that the first term on the right-hand side converges to :
$$
\int_{\Omega} j(x) \cdot w(x){\rm d}x.
$$
To compute the limit of the remaining term, we introduce:
$$
\sigma^N  =\dfrac{1}{ N |\lambda^N|^3} \sum_{\kappa \in \Kappa^N} \left[\sum_{i \in \mathcal I_{\kappa}^N} w(h_i^N)\right] \mathbf{1}_{T_{\kappa}^N}\,,
$$
so that:
$$
\dfrac{1}{N} \sum_{\kappa \in \Kappa^N} \left[ \sum_{i\in \mathcal I_{\kappa}^N}  w(h_i^N)\right] \cdot \bar{u}_{\kappa}^N = \int_{\Omega} 
\sigma^N \cdot u^N(x){\rm d}x. 
$$
For $w \in C^{\infty}_c(\Omega),$ we have that $\sigma^N$ is bounded in $L^1(\Omega)$ and, under assumption \eqref{eq_ass4}, it
converges to $\sigma = \rho w$ in  $\mathcal D'(\Omega).$ However, this is not sufficient to compute the limit of this last term. Indeed we have
strong convergence of the sequence $u^N$ in $L^q(\Omega)$ for $q< 6$ only. Consequently, we need the supplementary assumption  
\eqref{eq_ass_concentration} which entails that $\sigma^N$ is bounded in $L^{\infty}(\Omega).$ Now, 
$\sigma^N$ converges in $L^{q}(\Omega) -w$ for arbitrary  $q \in (1,\infty)$ (up to the extraction of a subsequence) and 
combining this fact with the strong convergence of $u^N$ we obtain that:
$$
 \lim_{N\to \infty}\dfrac{1}{N} \sum_{\kappa \in \Kappa^N} \sum_{i\in \mathcal I_{\kappa}^N}  w(h_i^N) \cdot \bar{u}_{\kappa}^N =   \int_{\Omega} \rho(x) w(x) \cdot \bar{u}(x){\rm d}x. 
$$ 
This would end the proof if we could actually define $\bar{u}_{\kappa}^N$ as in  \eqref{eq_ukappa1} and prove that it induces a small error
by replacing $u^N$ with the average $\bar{u}_{\kappa}^N$ in the integral on $\partial T_{\kappa}^N.$ Unfortunately, for this, we need that the combination of stokeslets to which $u^N$ 
is multiplied is a solution to the Stokes equations on the set where the average is taken (in particular we cannot choose $T_{\kappa}^N$ here contrary to what we have written in \eqref{eq_ukappa1}). 
So, we introduce a parameter $\delta$ (which will  be large), we "delete" the holes in a $\lambda^N/\delta-$neighborhood of $\partial T_{\kappa}^N$ and we  construct $\bar{u}^{N}_{\kappa}$  as the average of $u^N$ on the $\lambda^N/(2\delta)$-neighborhood of 
$\partial T_{\kappa}^N$ (inside $T_{\kappa}^N$). By a suitable choice of the covering $(T_{\kappa}^N)_{\kappa \in \Kappa^N}$ we prove that the cost of this deletion process is  $O(1/\sqrt{\delta}).$ This relies on the two fundamental properties of our choice for the sets on which we average $u^N$: they are all obtained from a model annulus by translation and dilation, the non-deleted holes are "far" from this set (with respect to the decay of solutions to Stokes problems in exterior domains). Hence, we obtain that:
$$
\left|I_w - 6\pi a \int_{\Omega} (j(x) - \rho(x) \bar{u}(x)) \cdot w(x) {\rm d}x \right| \lesssim \dfrac{1}{\sqrt{\delta}}
$$ 
for arbitrary large $\delta.$ 

\medskip

To conclude, we mention that the limitations \eqref{eq_ass_lambdaN} on the sequence $(\lambda^N)_{N\in\mathbb N}$ have two different origins. First, solutions to the Stokes prolem in the perforated cubes $T_{\kappa}^N \setminus \bigcup_{i\in \mathcal I_{\kappa}^N} \overline{B_i^N}$ have to be close to a combination of stokeslet like \eqref{eq_wreplace}. Second, the deletion process that we depicted above must not be too expensive. In order to compute a sufficiently sharp bound on this error, we must replace again a modified test-function $\bar{w}$ by a suitable combination of stokeslet. It turns out that the combination \eqref{eq_wreplace}
is not optimal. We must adapt here ideas coming from the reflection method (see \cite{HofVelpp} and the references therein).

\medskip

\subsection{Notations}\label{sec_notations} In the whole paper, for arbitrary $x \in \mathbb R^3$ and $r>0,$ we denote $B_{\infty}(x,r)$ the open ball with center $x$ and radius $r$
for the $\ell^{\infty}$ norm. The classical euclidean balls are denoted $B(x,r).$ For $x \in \mathbb R^3$  and $0 <\lambda_1< \lambda_2 $ we also denote:
$$
A(x,\lambda_1,\lambda_2) := B_{\infty}(x,\lambda_2) \setminus \overline{B_{\infty}(x,\lambda_1)}\,.
$$ 
The operator distance (between sets) is always computed with the $\ell^{\infty}$ norm.
We constantly use scaled truncation functions.
A first family of  truncation functions is constructed in a classical way. We introduce $\chi \in C^{\infty}_c(\mathbb R^3)$ such that $\chi = 1$ on $[-1,1]^3$ and 
$\chi = 0$ outside $[-2,2]^3.$ For arbitrary $\sigma>0,$ we denote $\chi_{\sigma} = \chi(\cdot/\sigma)$ its rescaled versions. This truncation function satisfies :
\begin{itemize}
\item $\chi_{\sigma} = 1$ on $B_{\infty}(0,\sigma)$ and $\chi_{\sigma} = 0$ outside $B_{\infty}(0,2\sigma),$ \\[-8pt]
\item $\nabla \chi_\sigma$ has support in $A(0,\sigma,2\sigma)$ and size $O(1/\sigma).$
\end{itemize}
The second family is denoted $\zeta_{\delta} \in C^{\infty}(\mathbb R^3)$ with a parameter $\delta > 0$ and satisfies:
$$
\zeta_{\delta}(x) = 0 \text{ in $B_{\infty}\left(0, \frac{1- 1/{\delta}}{2}\right)$ } \quad \text{ and } \quad \zeta_{\delta}(x) = 1 \text{ outside $B_{\infty}\left(0,\frac 12\right)$}.
$$

When we truncate vector-fields with $\chi_\sigma$ or $\zeta_{\delta}$, we create {\em a priori} non divergence-free vector-fields. To lift the divergence of the truncated vector-fields, we use
extensively the Bogovskii operator $\mathfrak B_{x,\lambda_1,\lambda_2}$ on the "cubic" annulus $A(x,\lambda_1,\lambda_2)$ (again $x \in \mathbb R^3$ and $0 < \lambda_1 < \lambda_2$). 
We recall that $w= \mathfrak{B}_{x,\lambda_1,\lambda_2}[f]$ is defined for arbitrary $f \in L^2(A(x,\lambda_1,\lambda_2)),$ whose mean vanishes, and yields a 
$H^1_0(A(x,\lambda_1,\lambda_2))$ vector-field such that ${\rm div} \, w = f.$ As the returned vector-field vanishes on $\partial  A(x,\lambda_1,\lambda_2) $
we extend it  by $0$ to obtain a $H^1(\mathbb R^3)$ function. We refer the reader to \cite[Section III.3]{Galdi} for more details on the divergence problem and the Bogovskii operator.

\medskip 
 
For legibility we also make precise a few conventions. We have the following generic notations:
\begin{itemize}
\item $u$ is a velocity-field solution to a Stokes problem, with associated pressure $p$,
\item $w$ is a data/test-function,
\item $I$  is an integral while $\mathcal I$ is a set of indices,
\item $T$ is a cube, depending on the width we shall use different exponents,
\item $n$  denotes the outward normal to the open set under consideration\,.
\end{itemize}
We shall also use extensively the symbol $\lesssim$ to denote that we have an inequality with a non-significant constant.
We mean that we denote $a \lesssim b$ when there exists a constant $C$ -- which is not relevant in the calculation --
such that $a \leq C b.$

\medskip

\subsection{Outline of the paper}
As our proof is based on fine properties of the Stokes problem, we recall in next section basics and advanced material on the resolution of this problem in bounded domains, in exterior domains and in a model cell domain. The core of the paper is sections \ref{sec_prfpart1} and \ref{sec_pfpart2} where a more rigorous statement of our main result is given and the proof is developed.
In a concluding section, we provide some remarks and examples on the optimality/limits of our dilution assumptions.
Finally,  we collect in two appendices technical properties on the Bogovskii operators, Poincar\'e-Wirtinger inequalities and covering arguments in measure theory.

\section{Analysis of the Stokes problem} \label{sec_Stokes}
In this section, we provide technical results on the resolution of the Stokes  problem:
\begin{equation} \label{eq_stokes}
\left\{
\begin{array}{rcl}
- \Delta u + \nabla p &=& 0\,, \, \\
{\rm div}\, u &= & 0 \,,
\end{array}
\right.
\quad \text{ on $\mathcal F$}\,,
\end{equation}
completed with boundary conditions 
\begin{equation} \label{cab_stokes}
u = u_* \,, \quad \text{ on $\partial \mathcal F\,,$}
\end{equation}
for a lipschitz domain $\mathcal F$ and boundary condition $u_* \in H^{\frac 12}(\partial \mathcal F).$
We consider the different cases : $\mathcal F$ is a bounded set, an exterior domain,
or a perforated cube. In the second case, we complement the system with a vanishing condition at infinity.

\subsection{Reminders on the Stokes problem in a bounded or an exterior domain}
We first assume that $\mathcal F$ is  a bounded domain with a lipschitz boundary $\partial \mathcal F$.
In this setting, a standard way to solve the Stokes problem \eqref{eq_stokes}-\eqref{cab_stokes} is to work with a generalized formulation (see \cite[Section 4]{Galdi}). For this, we introduce:
 $$
D(\mathcal F) := \left\{u \in H^1(\mathcal F) \text{ s.t. }  {\rm div}\, u =0  \right\}\,,\quad 
D_0(\mathcal F) := \left\{u  \in H^1_0(\mathcal F) \text{ s.t. } {\rm div}\, u =0   \right\}\,.
$$
By \cite[Theorem III.4.1]{Galdi}, we have that $D_0(\mathcal F)$ is the closure for the $H^1_0(\Omega)-$norm  of 
$$
\mathcal D_0(\mathcal F) =  \left\{ w \in C^{\infty}_c(\mathcal F) \text{ s.t. }{\rm div} \,  w = 0\right\}\,.  \\
$$

\medskip

We have then the following definition 
\begin{definition}
Given  $u_* \in H^{\frac 12}(\partial \mathcal F),$ a vector-field $u \in D(\mathcal F)$ is called generalized solution to \eqref{eq_stokes}-\eqref{cab_stokes} if 
\begin{itemize}
\item $u = u_*$ on $\partial \mathcal F$ in the sense of traces,\\[-8pt]
\item for arbitrary $w \in D_0(\mathcal F),$ there holds:
\begin{equation} \label{eq_wfstokes}
\int_{\mathcal F} \nabla u : \nabla w = 0\,.
\end{equation}
\end{itemize}
\end{definition} 
This generalized formulation  is obtained assuming that we have a classical solution, multiplying \eqref{eq_stokes} with arbitrary $w \in \mathcal D_0(\mathcal F)$
and performing integration by parts. De Rham theory ensures that conversely, if one constructs a generalized solution then it is possible to find a pressure $p$
such that \eqref{eq_stokes} holds in the sense of distributions.
Standard arguments yield:
\begin{theorem} \label{thm_varcar}
Assume that the boundary of the fluid domain $\partial \mathcal F$ splits into $(N+1) \in \mathbb N$ lipschitz connected components $\Gamma_0,\Gamma_1,\ldots,\Gamma_N.$ 
Given $u_* \in H^{\frac 12}(\partial \mathcal F)$ satisfying 
\begin{equation} \label{eq_noflux}
\int_{\Gamma_i}  u_* \cdot n {\rm d}\sigma = 0\,, \quad \forall \, i \in \{0,\ldots, N\},
\end{equation}
then 
\begin{itemize}
\item there exists a unique generalized solution $u$ to \eqref{eq_stokes}-\eqref{cab_stokes};
\item this generalized solution realizes
\begin{equation} \label{eq_varcar}
\inf \left\{ \int_{\mathcal F} |\nabla v|^2 , v \in D(\mathcal F) \text{ s.t. } v_{|_{\partial \mathcal F}} = u_* \right\}\,.
\end{equation}
\end{itemize}
\end{theorem}
\begin{proof}
Existence and uniqueness of the generalized solution is a consequence of \cite[Theorem IV.1.1]{Galdi}. 
A key argument in the proof of this reference is the property of traces that we state in the following lemma:
\begin{lemma} \label{lem_trace}
For arbitrary  $u_* \in H^{\frac 12}(\partial \mathcal F)$ satisfying \eqref{eq_noflux} there holds:
\begin{itemize}
\item  there exists $u_{bdy} \in D(\mathcal F)$ having trace $u_*$ on $\partial \mathcal F$\,,\\[-8pt]
\item  for arbitrary $u_{bdy} \in D(\mathcal F)$ having trace $u_*$ on $\partial \mathcal F$ there holds
$$
\left\{ u \in D(\mathcal F) \text{ s.t. } u_{|_{\partial \mathcal F}} = u_* \right\} = u_{bdy} + D_0(\mathcal F)\,.
$$
\end{itemize}
\end{lemma}
Then, given $u\in D(\mathcal F)$ the generalized solution to \eqref{eq_stokes}-\eqref{cab_stokes}  and $w \in D_0(\mathcal F),$
the fundamental property \eqref{eq_wfstokes} of $u$ entails that:
\begin{eqnarray*}
\int_{\mathcal F} |\nabla (u+w)|^2 &=& \int_{\mathcal F} |\nabla u|^2 + 2 \int_{\mathcal F} \nabla u : \nabla w + \int_{\mathcal F} |\nabla w|^2 \,,\\
										&=&\int_{\mathcal F} |\nabla u|^2 + \int_{\mathcal F} |\nabla w|^2 \,.
\end{eqnarray*}
Consequently, the norm on the left-hand side is minimal if and only if $w =0.$ Combining this remark with the above lemma yields that the generalized solution to \eqref{eq_stokes}-\eqref{cab_stokes} is the unique vector-field realizing \eqref{eq_varcar}.
\end{proof}

As mentioned previously, once it is proven that there exists a unique generalized solution $u$ to \eqref{eq_stokes}-\eqref{cab_stokes}, it is possible
to recover a pressure $p$ so that \eqref{eq_stokes}-\eqref{cab_stokes} holds  in the sense of distributions. If the data are smooth ({\em i.e.}  $\mathcal F$ has smooth boundaries and $u_*$ is smooth) one proves also that $(u,p) \in C^{\infty}(\overline{\mathcal F}).$ 

\medskip

We turn to the exterior problem as developed in \cite[Section 5]{Galdi}.  We assume now that $\mathcal F = \mathbb R^3 \setminus \overline{B^a}$ where $B^a= B(0,a)$ 
and we consider the Stokes problem \eqref{eq_stokes} with boundary condition 
\begin{equation} \label{cab_stokesexterior}
u= u_* \text{ on $\partial B^{a}$}\,, \quad \lim_{|x| \to \infty} u(x) = 0\,,
\end{equation}
for some $u_* \in H^{\frac 12}(\partial B^a).$ 
For the exterior problem, we keep the definition of generalized solution up to change a little the function spaces.
We denote in this case:
\begin{itemize}
\item $\mathcal D(\mathcal F) = \left\{ w_{|_{\mathcal F}}\,, \; w \in C^{\infty}_c(\mathbb R^3) \text{ s.t. } {\rm div}\, w = 0 \right\}\,,$\\[-8pt]
\item $D(\mathcal F)$ is the closure of $\mathcal D(\mathcal F)$ for the norm:
$$
\|w\|_{D(\mathcal F)} = \left( \int_{\mathcal F} |\nabla w|^2 \right)^{\frac 12}.
$$ 
\end{itemize}
We keep the definition of $\mathcal D_0(\mathcal F)$ as in the bounded-domain case and we construct $D_0(\mathcal F)$
as the closure of $\mathcal D_0(\mathcal F)$ with respect to this latter homogeneous $H^1$-norm. We note that, in the exterior
domain case, we still have that $D(\mathcal F) \subset W^{1,2}_{loc}(\mathcal F)$ (see \cite[Lemma II.6.1]{Galdi}) so that we have a trace operator on $\partial B^a$
and an equivalent to {Lemma \ref{lem_trace}}. 

\medskip

As in the case of bounded domains, the Stokes problem \eqref{eq_stokes}-\eqref{cab_stokesexterior} with boundary conditions
$u_*$ prescribing no flux through $\partial B^a$ has a  unique generalized solution (see \cite[Theorem V.2.1]{Galdi}, actually this existence/uniqueness result does not require the no-flux assumption) that satisfies a minimization problem. Thus, this solution satisfies:
\begin{itemize}
\item $\nabla u \in L^2(\mathbb R^3 \setminus \overline{B^a})\,,$ \\[-8pt]
\item for any $w \in D_0(\mathbb R^3 \setminus \overline{B^a})$ there holds:
$$
\int_{\mathbb R^3 \setminus \overline{B^a}} \nabla u : \nabla  w = 0,
$$
\item $u$ realizes:
\begin{equation} \label{eq_carvar_exterior}
\inf \left\{ \int_{\mathbb R^3 \setminus \overline{B^a}} |\nabla v|^2, v \in D(\mathbb R^3 \setminus \overline{B^a}) \quad v_{|_{\partial B^a}} = u^*\right\}.
\end{equation}
\end{itemize}
Explicit formulas  are provided when the boundary condition $u_*= v$ with $v \in \mathbb R^3$ constant (see \cite[Section 6.2]{DGR} for instance):
\begin{eqnarray}
u(x) = U^a[v](x) &:=& \dfrac{a}{4} \left( \dfrac{3}{|x|} +  \dfrac{a^2}{|x|^3} \right) v + \dfrac{3a}{4}  \left( \dfrac{1}{|x|} - \dfrac{a^2}{|x|^3}  \right)  \dfrac{v \cdot x}{|x|^2} x \,,  \label{eq_stokeslet}\\
p(x) = P^a[v](x) &:=& \dfrac{3a}{2} \dfrac{v \cdot x}{|x|^3} \,. \label{eq_presslet}
\end{eqnarray}
We call this classical solution stokeslet in what follows. With these explicit formulas, we remark that:
\begin{equation} \label{eq_Stokesexpansion}
|U^a[v] (x)| \lesssim \dfrac{a|v|}{|x|} \,,  \quad |\nabla U^a[v](x)| + |P^a[v] (x)| \lesssim \dfrac{a|v|}{|x|^2}\,, \quad \forall \, x \in \mathbb R^3 \setminus \overline{B^a}\,.
\end{equation}
We recall also the "Stokes' law" for the force exerted by the flow on $\partial B^a$:
\begin{equation} \label{eq_stokesformula}
\int_{\partial B^a} (\partial_n U^a[v] - P^a[v] n) {\rm d}\sigma = -6\pi a v.
\end{equation}
For convenience, we extend the stokeslet $U^a[v]$  by $U^a[v] = v$ on $B^a$.

\medskip

In the more general case of a smooth boundary condition $u^*$ prescribing no flux on $\partial B^a,$ the variational characterization 
of the generalized solution to the Stokes problem \eqref{eq_carvar_exterior} entails the following lemma:

\begin{proposition} \label{prop_crudebound}
There exists a universal constant $K$ such that, given  a divergence-free 
vector-field $w^* \in C^{\infty}(B_{\infty}(0,2a))$, denoting $u^* = w^*_{|_{\partial B^a}}$ and  $u \in D(\mathbb R^3 \setminus \overline{B^a})$ the unique generalized solution to \eqref{eq_stokes}-\eqref{cab_stokesexterior}, we have:
$$
\|u\|_{D(\mathbb R^3 \setminus \overline{B^a})} \leq K\sqrt{a} \left(\|w^*\|_{L^{\infty}(B_{\infty}(0,2a))} + a \|\nabla w^*\|_{L^{\infty}(B_{\infty}(0,2a))} \right) .
$$
\end{proposition}
\begin{proof}
Following the variational characterization of $u,$ the main point of the proof is the construction of a suitable lifting of $u^*.$
We set:
$$
\bar{u} = \chi_a w^* - \mathfrak{B}_{0,a,2a} [{\rm div}(\chi_a w^*)].
$$
Since $w^*$ is smooth and divergence free, this construction yields a divergence-free vector field $\bar{u} \in H^1_0(B(0,2a)),$
such that $\bar{u} = w^*$ on $\partial B^a.$ We have then:
\begin{eqnarray*}
\|u\|_{D(\mathbb R^3 \setminus \overline{B^a})} &\leq & \|\nabla \bar{u}\|_{L^2(\mathbb R^3 \setminus \overline{B^a})} \\
				&\leq & \|\nabla \chi_a w^*\|_{L^2(\mathbb R^3 \setminus \overline{B^a})} + \|\nabla \mathfrak{B}_{0,a,2a} [{\rm div}(\chi_a w^*)]\|_{L^2(\mathbb R^3 \setminus \overline{B^a})}.
\end{eqnarray*}
Since ${\rm div} w^* = 0,$ applying {Lemma \ref{lem_div}} on the Bogovskii operator yields a constant $K_0$ independant of $a$ such that:
$$
\|\nabla \mathfrak{B}_{0,a,2a} [{\rm div}(\chi_a w^*)]\|_{L^2(\mathbb R^3 \setminus \overline{B^a})} \leq K_0 \|w^* \cdot \nabla \chi_a \|_{L^2(\mathbb R^3 \setminus \overline{B^a})}.
$$
We conclude by computing explicitly:
\begin{eqnarray*}
\|w^* \cdot \nabla \chi_a \|_{L^2(\mathbb R^3 \setminus \overline{B^a})}   + \|\chi_a \nabla w^*\|_{L^2(\mathbb R^3 \setminus \overline{B^a})}.
\end{eqnarray*}

\end{proof}

\subsection{Stokes problem in a perforated cube} \label{sec_truncationprocess}
In this last subsection, we fix $M \in \mathbb N \setminus \{0\}$ together with $(a,\lambda) \in (0,\infty)^2$ and a divergence-free $w \in C^{\infty}_c(\mathbb R^3)$.
We consider the resolution of the Stokes problem in a cube of width $\lambda$ perforated with $M$ spherical holes of radius $a$ on which the velocity-field
$w$ is imposed. So, we fix $x_0 \in \mathbb R^3,$ we  denote $T = B_{\infty}(x_0,\lambda/2)$ an open cube of width $\lambda,$  and $B_i = B(h_i,a) \subset T$ for  $i=1,\ldots,M.$

\medskip

To state the main result of this subsection, we introduce two parameters: $d_m \in (0,\infty)$ is small while $\delta \in (0,\infty)$ is large. We assume that: 
\begin{equation}
d_m \leq \min_{i =1,\ldots,M} \min_{j \neq i}  |h_i - h_j |\,,   \qquad  \dfrac{\lambda}{\delta} \leq \min_{i=1,\ldots,M} \text{dist}(h_{i}, \partial T)  \,, 
\end{equation}
with
\begin{equation} \label{eq_hypdm}
\min \left( d_{m} , \frac{\lambda}{\delta} \right)  > 4a \,.
\end{equation}
We consider then the Stokes problem: 
\begin{equation} \label{eq_stokesledeux}
\left\{
\begin{array}{rcl}
- \Delta u + \nabla p &=& 0\,, \, \\
{\rm div}\,  u &= & 0 \,,
\end{array}
\right.
\quad \text{ on $\mathcal F = T \setminus \bigcup_{i=1}^M \overline{B_i}$}\,,
\end{equation}
completed with boundary conditions 
\begin{equation} \label{cab_stokesledeux}
\left\{
\begin{array}{rcll}
u(x) &=& w(x)\,, &\text{ on $\partial B_i\,,$ } \forall \, i =1,\ldots, M\,,\\[4pt]
u(x) &=& w(x)\,, & \text{ on $\partial T\,.$}
\end{array}
\right.
\end{equation}
Assumption \eqref{eq_hypdm} entails that the $B_i$ do not intersect and do not meet the boundary $\partial T.$ So, the set
$T \setminus \bigcup_{i=1}^M \overline{B_i}$ has a lipschitz boundary that one can decompose into $M+1$ connected components
corresponding to $\partial T$ and $\partial B_i$ for $i=1,\ldots,M.$ 
Direct computations show that:
\begin{align*}
& \int_{\partial B_i} w \cdot n {\rm d}\sigma = \int_{B_i} {\rm div} \, w =0\,, \qquad \text{ for $i=1,\ldots,M,$}
\\
& \int_{\partial T} w \cdot n {\rm d}\sigma = \int_{T} {\rm div} \, w =0.
\end{align*}
Hence, the problem \eqref{eq_stokesledeux}-\eqref{cab_stokesledeux} is solved by applying Theorem \ref{thm_varcar}
and it admits a unique generalized solution $u \in H^1(\mathcal F).$  

\medskip

A first crude bound on $u$ can be computed by adapting the proof of {Proposition \ref{prop_crudebound}}.
This yields:
\begin{proposition} \label{prop_crudeboundsvl}
Under the assumption \eqref{eq_hypdm}, there exists a constant $K_0$ independant of $(M,d_m,w,a,\lambda,\delta)$ and a constant $C_{\delta}$ depending only on $\delta$ such that:
\begin{eqnarray*}
\|\nabla u\|_{L^2(\mathcal F)} &\leq& K_0 \sqrt{Ma} \left( \max_{i=1,...,M}  \|w\|_{L^{\infty}(B_{\infty}(h_i,2a)} + a\max_{i=1,...,M}  \|\nabla w\|_{L^{\infty}(B_{\infty}(h_i,2a)}\right) \\
&& \quad + C_{\delta}  \|\nabla w\|_{L^2(\mathbb R^3 \setminus B_{\infty}(x_0,[1-1/(4\delta)]\lambda/2))} .
\end{eqnarray*}
\end{proposition}
\begin{proof}
Similarly to the Proposition \ref{prop_crudebound}, under the assumption \eqref{eq_hypdm}, we may 
construct a lifting of the boundary condition \eqref{cab_stokesledeux} by patching together liftings around the $B_i:$
\begin{eqnarray*}
\bar{u} &=& \sum_{i=1}^M \chi_a(\cdot-h_i) w - \mathfrak{B}_{h_i,a,2a} [x \mapsto {\rm div}(\chi_a(x-h_i) w(x))] \\
&& \quad + \zeta_{4\delta} ((\cdot-x_0)/\lambda) w - \mathfrak B_{x_{0},[1 - 1/(4\delta)]\lambda/2,\lambda/2} \left[ x \mapsto {\rm div} ( \zeta_{4\delta}((x-x_0)/\lambda)w(x)) \right] \,.
\end{eqnarray*}
We recall here that $x_0$ is the center of $T$ while $\chi_a,\zeta_{4\delta}$ are the truncation functions that we introduce in Section \ref{sec_notations}. Combining the variational characterization of $u$ with computations that are similar to the proof of Proposition \ref{prop_crudebound} entail
the result. % 
We only detail the control of the term on the second line:
\[
\bar{u}_{ext} =  \zeta_{4\delta} ((\cdot-x_0)/\lambda) w - \mathfrak B_{x_{0},[1 - 1/(4\delta)]\lambda/2,\lambda/2} \left[ x \mapsto {\rm div} ( \zeta_{4\delta}((x-x_0)/\lambda)w(x)) \right]
\] 
Applying the properties of the Bogovskii operator, we have:
\begin{align*}
\|\nabla \bar{u}_{ext}\|_{L^2(\mathcal F)}
& \leq C_{\delta} \Bigl[ \|(x-x_0)\nabla \zeta_{4\delta} ((\cdot-x_0)/\lambda)\|_{L^{\infty}(\mathcal F)} \|w(x)/|(x-x_0)|\|_{L^2(\mathbb R^3 \setminus B_{\infty}(x_0,[1-1/(4\delta)]\lambda/2))} \\
& \quad + \|\nabla w\|_{L^2(\mathbb R^3 \setminus B_{\infty}(x_0,[1-1/(4\delta)]\lambda/2))} \Bigr].
\end{align*}
We apply then the Hardy inequality in exterior domains (see the proof of \cite[Theorem II.6.1-(i))]{Galdi}) to bound:
\[
\|w(x)/(x-x_0)\|_{L^2(\mathbb R^3 \setminus B_{\infty}(x_0,[1-1/(4\delta)]\lambda/2))}  \leq C_{\delta} \|\nabla w\|_{L^2(\mathbb R^3 \setminus B_{\infty}(x_0,[1-1/(4\delta)]\lambda/2))}. 
\]
We note that the constant appearing here is independant of $\lambda$ by a scaling argument.
\end{proof}

In what follows, we focus on the case where $w$ vanishes on $\partial T.$ We look for  more detailed informations on $u.$ In particular, we want to compare the solution $u$  with combinations of stokeslets:
$$
\sum_{i=1}^M U^a[w_i](x-h_i).
$$
Here $(w_1,\ldots,w_M) \in [\mathbb R^3]^M$ are to be chosen and $U^a$ is  defined in \eqref{eq_stokeslet}.
In this respect, our first main result reads:
\begin{proposition} \label{prop_truncationprocess} 
Let assume further that $w \in C^{\infty}_c(T)$ and denote:
$$
u_s(x) = \sum_{i=1}^M U^a[w(h_i)](x-h_i) \qquad \forall \, x \in \mathbb R^3.
$$
There exists a constant $K_0$ independent of $(M,d_m,w,a,\lambda,\delta)$
and a constant  $C_{\delta}$ depending only on $\delta$  for which:
$$
\|\nabla (u-u_s)\|_{L^2(\mathcal F)}  
\leq K_0 \|w\|_{W^{1,\infty}(\mathbb R^3)}  \left( \sqrt{Ma^{3}}\left[  1 +  \dfrac{M^{2/3}}{d_{m}} +  \dfrac{a M^{1/3}}{d_{m}^2}    \right] +C_{\delta}\dfrac{Ma}{\sqrt{\lambda}}\right)\,.  
$$
\end{proposition}
\begin{proof}
We split the error term into two pieces. First, we reduce the boundary conditions of the Stokes problem \eqref{eq_stokesledeux}-\eqref{cab_stokesledeux} to constant
boundary conditions. Then, we compare the solution to the Stokes problem with constant boundary conditions to the combination of stokeslets 
$u_s.$ In the whole proof, the symbol $\lesssim$ is used when the implicit constant in our inequality does not depend on $M,d_m,w$ and $a,\lambda,\delta.$

\medskip

So, we introduce  ${u}_c$ the unique generalized solution to the Stokes problem on $\mathcal F$ with boundary conditions:
\begin{equation} \label{cab_stokesletrois}
\left\{
\begin{array}{rcll}
u_c &=& w(h_i)\,, & \text{ on $\partial B_i\,,$ } \forall \, i =1,\ldots, M\,,\\[4pt]
u_c &=& 0\,, & \text{ on $\partial T\,.$}
\end{array}
\right.
\end{equation}
Again, existence and uniqueness of this velocity-field holds by applying Theorem \ref{thm_varcar}.
We split then:
\begin{eqnarray*}
\|\nabla (u-u_s)\|_{L^2(\mathcal F)} &\leq& \|\nabla (u - u_c)\|_{L^2(\mathcal F)} + \|\nabla (u_c-u_s)\|_{L^2(\mathcal F)}.
\end{eqnarray*}
To control the first term on the right-hand side, we note that $(u-u_c)$ is the unique generalized solution to the Stokes problem on $\mathcal F$
with boundary conditions:
$$
\left\{
\begin{array}{rcll}
(u - u_c) (x) &=& w(x) - w(h_i)\,, & \text{ on $\partial B_i\,,$ } \forall \, i =1,\ldots, M\,,\\[4pt]
(u - u_c)(x) &=& 0\,, & \text{ on $\partial T\,.$}
\end{array}
\right.
$$
Hence, Proposition \ref{prop_crudeboundsvl} applies to $(u-u_c).$ This entails that:
$$
\|\nabla (u-u_c)\|_{L^2(\mathcal F)} \lesssim \sqrt{Ma} \left[\max_{i=1,\ldots,M} \|w-w(h_i)\|_{L^{\infty}(B_{\infty}(h_i,2a))} + a \max_{i=1,\ldots,M} \|\nabla w\|_{L^{\infty}(B_{\infty}(h_i,2a))} \right] .
$$
Explicit computations yield eventually that:
\begin{equation} \label{eq_contu-uc}
\|\nabla (u-u_c)\|_{L^2(\mathcal F)} \lesssim \sqrt{Ma^3}  \|w\|_{W^{1,\infty}}\,.
\end{equation}
% At this point, we argue that the straightforward extension of $u$ and $u_c$ (by $w$ and $w(h_i)$ on the $B_i$ respectively) satisfy $(u-u_c) \in H^1_0(T) \subset L^6(T)$ so that 
%\begin{eqnarray*}
%\|u - u_c\|_{L^6(\mathcal F)} &\leq & \|u - u_c\|_{L^6(T)} \lesssim   \|\nabla (u-u_c)\|_{L^2(T)} \\
%			              &\lesssim&   \left(  \|\nabla (u-u_c)\|^2_{L^2(\mathcal F)} + {Ma^3} \|w\|^2_{W^{1,\infty}} \right)^{\frac 12} \\
%			              & \lesssim & \sqrt{Ma^3} \|w\|_{W^{1,\infty}}.
%\end{eqnarray*}
%We emphasize that, by a scaling argument, the constant deriving from the embedding $H^1_0(T) \subset L^6(T)$ does not
%depend on $\lambda$ so that it is not significant to our problem.

\medskip

We turn to estimating $v := u_c - u_s.$ Due to the linearity of the Stokes equations, $v$ is the unique generalized solution to the Stokes equation
on $\mathcal F$ with boundary condition: 
\begin{equation} 
\left\{
\begin{array}{rcll}
v &=& w(h_i) - u_s \,, & \text{ on $\partial B_i\,,$ } \forall \, i =1,\ldots, M\,,\\[4pt]
v &=& -u_s \,, & \text{ on $\partial T\,.$}
\end{array}
\right.
\end{equation}
Hence, Proposition \ref{prop_crudeboundsvl} applies again to $v.$  By construction, we note that:
$$
\begin{array}{rcll}
 v(x) &=& - \sum_{j\neq i} U^{a}[w(h_j)](x-h_j),   & \text{ on $\partial B_i,$ for $i=1,\ldots,M$},\\[10pt]
 v(x) &=& -\sum_{j=1}^M  U^{a}[w(h_j)](x-h_j),  &  \text{ on $\partial T$}
\end{array}
$$
Hence, we may choose as extension of these boundary conditions, any divergence-free vector-field $\tilde{w} \in C^{\infty}_c(\mathbb R^3)$  that satisfies:
$$
\begin{array}{rcll}
 \tilde{w}(x) &=& - \sum_{j\neq i} U^{a}[w(h_j)](x-h_j),   & \text{ on $B_{\infty}(h_i,2a),$ for $i=1,\ldots,M$},\\[10pt]
 \tilde{w}(x) &=& -\sum_{j=1}^M  U^{a}[w(h_j)](x-h_j),  &  \text{ on $\mathbb R^3 \setminus B_{\infty}(x_0,[1-1/(4\delta)]\lambda/2)$}.
\end{array}
$$
We emphasize that it is possible to construt such an extension by adapting the ideas in the proof of Proposition \ref{prop_crudeboundsvl} since the $U^a$ prescribe no flux through hypersurfaces. 

In order to apply Proposition \ref{prop_crudeboundsvl}, we bound first
$$
\max_{i=1,\ldots,M} \|\tilde{w}\|_{L^{\infty}(B_{\infty}(h_i,2a))}\,,  \qquad \max_{i=1,\ldots,M} \|\nabla \tilde{w}\|_{L^{\infty}(B_{\infty}(h_i,2a))}\,.
$$
Given $i\in \{1,\ldots,M\}$, thanks to the asymptotic expansion of the stokeslet \eqref{eq_Stokesexpansion} and because $|x-h_i| \geq |h_i-h_j|/2$ on $B(h_i,2a)$ (recall \eqref{eq_hypdm}), we have:
$$
\|\tilde{w}\|_{L^{\infty}(B_{\infty}(h_i,2a))} \lesssim \sum_{j\neq i} \dfrac{a|w(h_j)|}{|h_i-h_j|} 
\qquad  
\|\nabla \tilde{w}\|_{L^{\infty}(B_{\infty}(h_i,2a))} \lesssim \sum_{j\neq i} \dfrac{a|w(h_j)|}{|h_i-h_j|^2} . 
$$
Applying \cite[Lemma 2.1]{OJ} to bound the sum on $j$ appearing above entails:
\begin{equation} \label{eq_tw_1}
\|\tilde{w}\|_{L^{\infty}(B_{\infty}(h_i,2a))} \lesssim \dfrac{M^{2/3}a}{d_{m}} \qquad \|\nabla \tilde{w}\|_{L^{\infty}(B_{\infty}(h_i,2a))} \lesssim \dfrac{M^{1/3}a}{d_{m}^2}.
\end{equation}

We turn now to compute a bound for:
$$
\|\nabla \tilde{w}\|_{L^2(\mathbb R^3 \setminus B_{\infty}(x_0,[1-1/(4\delta)]\lambda/2))}.
$$ 
To this end, we note that, by definition of $\tilde{w},$ we have:
$$
\|\nabla \tilde{w}\|_{L^2(\mathbb R^3 \setminus B_{\infty}(x_0,[1-1/(4\delta)]\lambda/2))} \leq \sum_{i=1}^M \|\nabla U^a[w(h_i)](\cdot-h_i)\|_{L^2(\mathbb R^3 \setminus B_{\infty}(x_0,[1-1/(4\delta)]\lambda/2))} .
$$
Then, given $i \in \{1,\ldots,M\},$ because of assumption \eqref{eq_hypdm}, we have 
 that ${\rm dist}(h_i, T \setminus B_{\infty}(x_0,[1-1/(4\delta)]\lambda/2)) \geq \lambda/8\delta.$ 
Replacing the stokeslet with its explicit value, we obtain thus:
\begin{eqnarray*}
\|\nabla U^a[w(h_i)](\cdot-h_i)\|_{L^2(\mathbb R^3 \setminus B_{\infty}(x_0,[1-1/(4\delta)]\lambda/2))}
&\lesssim & \left( \int^{\infty}_{\lambda/8\delta} \dfrac{a^2 |w(h_i)|^2}{r^2} {\rm d}r\right)^{1/2} \\
& \leq & \dfrac{C_{\delta}a|w(h_i)|}{\sqrt{\lambda}}
\end{eqnarray*}
After combination, we derive finally:
\begin{equation} \label{eq_tw_2}
\|\tilde{w}\|_{H^1(T \setminus B_{\infty}(x_0,[1-1/(4\delta)]\lambda/2))} \leq \dfrac{C_{\delta} M a \|w\|_{L^{\infty}}}{\sqrt{\lambda}}.
\end{equation}
Hence, applying Proposition \ref{prop_crudeboundsvl} to $v$ yields,  with the computations \eqref{eq_tw_1} and \eqref{eq_tw_2}, 
that:
\begin{equation} \label{eq_contuc-us}
\|\nabla (u_c-u_s)\|_{L^2(\mathcal F)} \lesssim \|w\|_{L^{\infty}} \left( \sqrt{Ma} \left[  \dfrac{aM^{2/3}}{d_{m}} +  \dfrac{a^2 M^{1/3}}{d_{m}^2} \right]+ \dfrac{C_{\delta} M a}{\sqrt{\lambda}} \right).
\end{equation}

%To conclude the proof, we extend $u_c$ by $0$ outside $T$ and by its boundary values on the $B_i.$ We have then that 
%$\nabla(u_c - u_s) \in L^2(\mathbb R^3)$ so that, by classical sobboleve embeddings: 
%$$
%\|u_c - u_s\|_{L^6(T)} \leq \|u_c - u_s\|_{L^6(\mathbb R^3)} \lesssim  \|\nabla (u_c-u_s)\|_{L^2(\mathbb R^3)}
%$$ 
%However, similarly to the previous computations on $\bar{v}_0$  we obtain:
%\begin{eqnarray*}
%\|\nabla (u_c-u_s)\|^2_{L^2(\mathbb R^3)} &=& \|\nabla (u_c-u_s)\|^2_{L^2(T)} + \|\nabla u_s\|^2_{L^2(\mathbb R^3 \setminus \overline{T})} \\
%								&\lesssim &  \|w\|^2_{L^{\infty}} \left( {Ma^3} \left[  \dfrac{M^{2/3}}{d_{min}} +  \dfrac{a M^{1/3}}{d_{min}^2} \right]^2+ C_{\delta}(M a)^2 \left[ \lambda + \dfrac{1}{\lambda}\right]  \right).
%\end{eqnarray*}
This ends up the proof.
\end{proof}
%
%
%We note from the proof that, if $w$ is locally constant on the $(B_i)_{i=1,\ldots,M},$ the first error term $u-u_c$ vanishes. 
%We have thus the corollary:
%\begin{corollary}
%If we assume further that $w$ is locally constant on $\cup_{i=1}^M B_i,$ there holds: 
%\begin{multline*}
%\|(u-u_s)\|_{L^6(\mathcal F)}   + \|\nabla (u-u_s)\|_{L^2(\mathcal F)}  
%\\
%\leq K \|w\|_{L^{\infty}(\mathbb R^3)}  \left( \sqrt{Ma^{3}}\left[  \dfrac{M^{2/3}}{d_{min}} +  \dfrac{a M^{1/3}}{d_{min}^2}    \right] +C_{\delta}Ma \left[ \lambda  + \frac 1\lambda \right]^{\frac 12}\right)\,.  
%\end{multline*}
%
%\end{corollary}
%

Choosing $w_i = w(h_i),$ the combination of stokeslet that we obtain (namely $u_s$) is not a sufficiently good approximation of $u$ for our later purpose. It turns out that the error term $aM^{2/3}/d_m$ is too large. Adapting the method of reflection of \cite{OJ} (see also \cite{HofVelpp}) we find a  better choice that yield an approximation error without this term.
The result of this analysis is the content of the following proposition:
\begin{proposition} \label{prop_comparestks2} 
Let denote $\xi := {M^{2/3}a}/d_{m} $ and assume that $\xi \leq \xi_{max} < 1.$
If $w \in C^{\infty}_c(T),$ there exists a constant $K_{max}$ depending only on $\xi_{max}$ and a constant  $C_{\delta}$ depending only on $\delta$  for which the following statements holds true:
\begin{itemize}
\item[i)] There exists $(w^{(\infty)}_1,\ldots,w^{(\infty)}_M) \in \mathbb R^3$ so that:
$$\max_{i=1,\ldots,M} | w^{(\infty)}_i - w(h_i)| \leq K_{max}\dfrac{M^{2/3}a}{d_{m}}
$$
\item[ii)] Denoting 
$
\bar{u}_s = \sum_{i=1}^M U^a[w^{(\infty)}_i](\cdot-h_i)
$ 
we have:
\begin{multline*}
\|\nabla  (u - \bar{u}_s)\|_{L^2(T)}  + \|u - \bar{u}_s\|_{L^6(T)}  \\ \leq K_{max}\|w\|_{W^{1,\infty}} \left( \sqrt{Ma^3} + \sqrt{Ma} \dfrac{M^{1/3}a^2}{d_{m}^2}+    \dfrac{C_{\delta} Ma}{\sqrt{\lambda}}\right) 
\end{multline*}
\end{itemize} 

\end{proposition}
\begin{proof}
We first remark that we may restrict to constant boundary conditions by introducing the solution $u_c$ as in the previous proof.
This yields an error term of size $\sqrt{Ma^3}\|w\|_{W^{1,\infty}}.$ Thus, our proof reduces to computing an approximation  
for the generalized solution $u_c$ to the Stokes problem with boundary conditions
\begin{equation} \label{cab_stokesletrois_bis}
\left\{
\begin{array}{rcll}
u_c &=& w(h_i)\,, & \text{ on $\partial B_i\,,$ } \forall \, i =1,\ldots, M\,,\\[4pt]
u_c &=& 0\,, & \text{ on $\partial T\,.$}
\end{array}
\right.
\end{equation}

{\em Step 1: Method of reflection.} 
Following the ideas of the method of reflection, we remark that to construct $\bar{u}_s$ the first natural try would be:
$$
u_s(x) = \sum_{i=1}^{M} U^a[w(h_i)](x-h_i).
$$
However, doing so, we create a solution to the Stokes equations  which does not match the right boundary conditions. 
Indeed, we have:
$$
u_s(x) = w(h_i) + \sum_{j\neq i} U^a[w(h_i)](x-h_j) \text{ on $\partial B_i$ for $i=1,\ldots,M$}.
$$
As in the case of non-constant boundary conditions, the idea is to approximate the error by a constant in order to improve
the approximation. This is the motivation for the following iterative process.
We define 
$$
w_i^{(0)} = w(h_i),  \quad  \text{ for $i=1,\ldots,M,$} \\
$$
Given $k\in \mathbb N,$ assuming that $(w_i^{(l)})_{i=1,\ldots,M}$ are constructed for $l \in \{0,\ldots,k\}$, we set:
$$
\left\{
\begin{array}{rcl}
s_i^{(k)}			&=& \displaystyle\sum_{l=0}^k w_{i}^{(l)}\,,\\[8pt]
w^{(k+1)}_i &=&   w(h_i) -  \displaystyle \sum_{i=1}^{M} U^a\left[s_i^{(k)} \right](h_i-h_j) \,,
\end{array}
\right.
\quad 
\text{ for $i=1,\ldots,M.$}
$$
This yields a sequence of $M-$uplets of vectors.

\medskip

Correspondingly, given $k \in \mathbb N,$ we set 
$$
u_s^{(k)}(x) = \sum_{i=1}^M U^a \left[s_{i}^{(k)}\right](x-h_i)\,,  \quad \forall \, x \in \mathbb R^3.
$$
We remark that, the above recursion formula also reads:
\begin{equation} \label{eq_recform}
w^{(k+1)}_i  = w(h_i)   - u_{s}^{(k)}(h_i)\,, \quad \text{ for $i=1,\ldots,M.$}
\end{equation}

{\em Step 2: Convergence.} We show now that the sequences  $(s_i^{(k)})_{k\in \mathbb N}$ converge for $i \in \{1,\ldots,M\}.$
This amounts to proving that $(w_i^{(k)})_{k\in \mathbb N}$ converges sufficiently fast to $0.$ 
To this end, we remark that, for $k \in \mathbb N,$ we have:
$$
u_{s}^{(k+1)}(x) = u_{s}^{(k)}(x) + \sum_{i=1}^M U^a \left[w_{i}^{(k)}\right](x-h_i) 
$$
Plugging this identity in the recursion formula \eqref{eq_recform} yields that, for $k\geq 1$ and $i=1,\ldots,M$ there holds:
\begin{eqnarray*}
w_{i}^{(k+1)} &=& (w(h_i) - u_{s}^{(k-1)}(h_i)) -  \sum_{j=1}^{M} U^{a}[w_{j}^{(k)}](h_i-h_j)  \\
						&=& w_i^{(k)} - \sum_{j=1}^{M} U^{a}[w_{j}^{(k)}](h_i-h_j)\\
						&=& -\sum_{j\neq i} U^{a}[w_{j}^{(k)}](h_i-h_j) .
\end{eqnarray*}
Applying again the asymptotics of the stokeslet with \cite[Lemma 2.1]{OJ} yields:
$$
\max_{i=1,\ldots,M} |w_{i}^{(k+1)}| \leq \dfrac{M^{2/3}a}{d_{m}} \max_{i=1,\ldots,M}|w_{i}^{(k)}|.
$$
Since, by assumption, we have $\xi = M^{2/3}a/d_{m} <1$, we obtain that $(s^{(k)}_i)_{k\in \mathbb N}$ converges and
that, denoting $w^{(\infty)}_i$ the limits, we have:
$$
|w_{i}^{(\infty)} - w_{i}^{(0)}| = |w_{i}^{(\infty)} - w(h_i)| \leq \|w\|_{L^{\infty}} \dfrac{\xi}{(1-\xi)} \leq \|w\|_{L^{\infty}} \dfrac{\xi}{(1-\xi_{max})}.
$$
We obtain $i).$

\medskip

{\em Step 3: Error estimate.} Since $(s_{i}^{(k)})$ converges, and $U^{a}[w] \in H^1_{loc}(\mathbb R^3)$ (whatever the value of $w \in \mathbb R^3$)  we also have that $u_{s}^{(k)}$ converges to $\bar{u}_s$ in $H^1(T),$ where:
$$
\bar{u}_{s}(x) = \sum_{i=1}^M U^a[w_{i}^{(\infty)}](x-h_i) .
$$
In particular, we have that
$$
\|\nabla (u_c - \bar{u}_s)\|_{L^2(T)} = \lim_{k \to \infty} \|\nabla (u_c - u_s^{(k)})\|_{L^2(T)} .
$$

Then, given $k \in \mathbb N,$ we remark that $v := u_c-u_s^{(k)}$ is the unique generalized solution to the Stokes problem with boundary conditions:
\begin{equation} %\label{cab_stokesletrois_bisbis}
\left\{
\begin{array}{rcll}
v &=& w(h_i) - u_s^{(k)}\,, & \text{ on $\partial B_i\,,$ } \forall \, i =1,\ldots, M\,,\\[4pt]
v &=& -u_s^{(k)}\,, & \text{ on $\partial T\,.$}
\end{array}
\right.
\end{equation}
Consequently, we apply once again {Proposition \ref{prop_crudeboundsvl}} 	to bound $v.$ For this, we remark that, since $U_a[s_i^{(k)}](x-h_i)$ is constant on $B_i$, there holds:
$$
v(x) = w(h_i) - u_s^{(k)}(h_i) - \sum_{j \neq i} \left[ U^a[s_j^{(k)}](x-h_j) - U^a[s_j^{(k)}](h_i-h_j)\right] \quad \forall \, x \in \partial B_i,
$$
where $w(h_i) - u_s^{(k)}(h_i)  = w_i^{(k+1)}.$ We keep notations $v$ for the extension of the right-hand side above. 
We can bound:
$$
\|v\|_{L^{\infty}(B_{\infty}(h_i,2a))} \lesssim |w_{i}^{(k+1)}| + a \|\sum_{j\neq i} \nabla U^a[s_j^{(k)}](\cdot-h_j)\|_{L^{\infty}(B_{\infty}(h_i,2a))}.
$$ 
To control the second term appearing above, we apply again the asymptotics of the stokeslets with \cite[Lemma 2.1]{OJ}
to obtain:
$$
\|\sum_{j\neq i} \nabla U^a[s_j^{(k)}](\cdot-h_j) \|_{L^{\infty}(B_{\infty}(h_i,2a))} \leq    K_{max} \|w\|_{L^{\infty}} \dfrac{M^{1/3}a}{d_{m}^2},
$$
and thus:
$$
\|v\|_{L^{\infty}(B_{\infty}(h_i,2a))} \lesssim |w_{i}^{(k+1)}| + K_{max}  \|w\|_{L^{\infty}} \dfrac{M^{1/3}a^2}{d_{m}^2}.
$$
As for the term on the boundary $\partial T$, we obtain with similar computations as the one to obtain \eqref{eq_tw_2} that:
\begin{equation} \label{labelmanquant}
\|\nabla u_s^{(k)}\|_{L^2(\mathbb R^3\setminus B_{\infty}(x_0,[1-1/(4\delta)]\lambda/2))} \leq \dfrac{C_{\delta}K_{max} Ma \|w\|_{L^{\infty}}}{\sqrt{\lambda}} .
\end{equation}
So, Proposition \ref{prop_crudeboundsvl} yields that:
\begin{eqnarray*}
\|\nabla (u_c-u_s^{(k)}) \|_{L^2(T)} &=& \|\nabla (u_c-u_s^{(k)}) \|_{L^2(\mathcal F)} \\
			&\leq & \sqrt{Ma} \left( \max_{i=1,\ldots,M} |w_{i}^{(k+1)}| +  K_{max}  \|w\|_{L^{\infty}} \dfrac{M^{1/3}a^2}{d_{m}^2}\right)  \\ && \quad +   \dfrac{C_{\delta}K_{max} Ma \|w\|_{L^{\infty}}}{\sqrt{\lambda}}.
\end{eqnarray*}
We extend then $u_c$ by $0$ outside $T.$ We have that 
$\nabla (u_c-u_s^{(k)}) \in L^2(\mathbb R^3).$ Classical Sobolev embedding yield that:
$$
\|u_c-u_s^{(k)}\|_{L^6(\mathbb R^3)} \lesssim  \|\nabla (u_c-u_s^{(k)}) \|_{L^2(\mathbb R^3)}\,,
$$
and, consequently:
$$
\|u_c-u_s^{(k)}\|_{L^6(T)} \lesssim 
\left( \|\nabla (u_c-u_s^{(k)}) \|_{L^2(T)}^2 +  \|\nabla u_s^{(k)} \|^2_{L^2(\mathbb R^3 \setminus T)}\right).
$$
Bounding the last term on the right-hand side as in \eqref{labelmanquant} (we note that only gradient terms appear), we obtain the similar bound for the $L^6$-norm:
\begin{eqnarray*}
\|u_c-u_s^{(k)}\|_{L^6(T)} &\lesssim& \sqrt{Ma} \left( \max_{i=1,\ldots,M} |w_{i}^{(k+1)}| +  K_{max}  \|w\|_{L^{\infty}} \dfrac{M^{1/3}a^2}{d_{m}^2}\right)  \\ && \quad +   \dfrac{C_{\delta}K_{max} Ma \|w\|_{L^{\infty}}}{\sqrt{\lambda}}.
\end{eqnarray*}
We conclude the proof of $ii)$ by taking the limit $k \to \infty.$

\end{proof}

\section{Proof of Theorem \ref{thm_main} -- Plan of the proof} \label{sec_uniform}

From now on, we fix a sequence of data $(v_i^{N},h_i^N)_{i=1,\ldots,N}$ associated with $(B_i^{N})_{i=1,\ldots,N}$ that satisfy \eqref{eq_ass0} for arbitrary $N \in \mathbb N$
and such that \eqref{eq_ass1}--\eqref{eq_ass5} hold true with
$$
j \in L^2({\Omega})  \,, \quad \rho \in L^{\infty}({\Omega})\,.
$$ 
%We introduce also $(\lambda^N)_{N\in \mathbb N} \in (0,\infty)^{\mathbb N}$ for which we have \eqref{eq_ass2}-\eqref{eq_ass3}.
Because of assumption \eqref{eq_ass0}, the existence result of the previous section applies so that there exists a unique generalized solution $u^N \in H^1(\mathcal F^N)$ 
to  \eqref{eq_stokesN}-\eqref{cab_stokesN}. In what follows, we extend implicitly $u^N$ by its boundary values on the $\partial B_i^N$:
$$
u^N = 
\left\{
\begin{array}{rl}
v_i^N\,, & \text{in $B_i^N,$ for $i=1,\ldots,N$}\,,\\[8pt]
u^N\,, & \text{in $\mathcal F^N.$}
\end{array}
\right.
$$
As the $B_i^N$ do not overlap and do not meet $\partial \Omega,$ it is straightforward that these velocity-fields
yield a sequence in $H^1_0(\Omega)$ of divergence-free vector-fields. Moreover, we have the property:
$$
\| \nabla u^N\|_{L^2(\mathcal F^N)} = \|  \nabla u^N\|_{L^2(\Omega)}.
$$
We also assume that \eqref{eq_ass_dmin}-\eqref{eq_ass_lambdaN} are in force. %
We have then:
\begin{align}
&\sup_{N \in \mathbb N} {Na^N} = a^{\infty} \in (0,\infty), \label{eq_ainfini} \\
%&\sup_{N \in \mathbb N} \; \dfrac{1}{|Nd_{min}^N|^{\frac 35}}  \; \sup_{x \in \Omega} \card \left\{ i \in \{1,\ldots,N \} \text{ s.t. } h_i^N \in \overline{B(x,\lambda^N)}\right\}  = M^{\infty} \label{eq_minfini}\\
& \sup_{N \in \mathbb N} \dfrac{1}{N} \sum_{i=1}^N |v_i^N|^2 = |\mathcal E^{\infty}|^{2} \in (0,\infty). \label{eq_einfini}
\end{align}
 
Our target result reads:
\begin{theorem} \label{thm_DGR}
The sequence of extended generalized solutions $(u^N)_{N \in \mathbb N}$ converges weakly in $H^1_0(\Omega)$ to $\bar{u}$ satisfying
\begin{itemize}
\item[(B1)\,] $\bar{u} \in H^1_0(\Omega)\,,$\\[-10pt]
\item[(B2)\,] ${\rm div}\, \bar{u} = 0$ on $\Omega\,,$\\[-10pt]
\item[(B3)\,] for any divergence-free $w \in {C}^{\infty}_c(\Omega)$  we have:
\begin{equation} \label{eq_SBweak}
\int_{\Omega} \nabla \bar{u}: \nabla w = {6\pi a}  \int_{\Omega}   \left[j - \rho \bar{u} \right] \cdot w \,.
 \end{equation}
 \end{itemize}
 
\end{theorem} 

\bigskip

Theorem \ref{thm_main} is a corollary of this theorem as (B1)-(B2)-(B3) corresponds to the generalized formulation of the Stokes-Brinkman system
\eqref{eq_SB}-\eqref{cab_SB}. The proof of this result is developed in the end of this section and the two next ones. 

\medskip

The scheme of the complete proof for Theorem \ref{thm_DGR} is as follows. We first obtain that the sequence $(u^N)_{N\in\mathbb N^*}$
is bounded in $H^1_0(\Omega).$ A straightforward consequence is that the sequence $(u^{N})_{N\in\mathbb N^*}$ is weakly-compact and we denote by $\bar{u}$ a cluster-point for the weak topology.

In the weak limit, $\bar{u}$ satisfies ${\rm div} \, \bar{u} =0$ on $\Omega.$ So $\bar{u}$ satisfies (B1) and (B2) of our theorem. 
The remainder of the proof consists in showing that 
it  satisfies (B3) also. Indeed, we remark that ${\rho}$ is the density of a probability measure. Hence $\rho \geq 0$  on $\Omega.$ 
By a simple energy estimate one may then show that, given $j \in L^2(\Omega),$ there exists at most one $\bar{u} \in H^1_0(\Omega)$ that satisfies 
simultaneously (B1)-(B2)-(B3). A direct corollary of this remark is that, if we prove that (B3) is satisfied by $\bar{u}$ we have uniqueness 
of the  possible cluster point to the sequence $(u^N)_{N\in \mathbb N}$ and the whole sequence converges to this $\bar{u}$ in $H^1_0(\Omega)-w.$

A classical method for obtaining (B3) is to fix a test-function $w$ and apply that
$$
\int_{\Omega} \nabla \bar{u}: \nabla w = \lim_{N\to \infty}\int_{\Omega} \nabla u^N : \nabla w \medskip 
$$
{(up to the extraction of a suitable subsequence of $(u^N)$)}. One may then want to use the equation satisfied by $u^N$ in order to rewrite the $N$-th integral in a way that induces the empirical measures $S_N.$
It would then be possible to apply the assumptions on the convergence of these empirical measures.
For this purpose, we fix a integer $\delta \geq 4,$ arbitrary large and we construct, for fixed $N,$ a suitable test-function ${w}^s$ (depending actually on $\delta$ and $N$) so that
\begin{itemize}
\item  by replacing $w$ with $w^s$ in $I^N$, we have:
$$
\int_{\Omega} \nabla u^N : \nabla w =\int_{\Omega} \nabla u^N : \nabla w^s + error^N
$$ 
with $error^N$ of size $1/{\delta}^{1/3}$ (plus terms depending on $N$ that vanish when $N \to \infty$),
\item when $N\to \infty$ we prove that, 
$$
\int_{\Omega} \nabla {u}^N :  \nabla w^s \to 6 \pi  a\int_{\Omega} (j - \rho \bar{u}) \cdot w + error\,,
$$
with an $error$ of size $1/\sqrt{\delta}.$
\end{itemize}
As $\delta$ can be taken arbitrary large, this yields the expected result. 

\medskip

A proof that $(u^N)_{N\in\mathbb N}$ is bounded is given in the end of this section. 
The construction of the modified test-function $w^s$ and the computation of the size of the error terms $error^N$
are provided in the next section. We complete the proof by computing the asymptotics of the integrals involving
$w^s$ and the computation of the error term $error$ in Section \ref{sec_pfpart2}.

\medskip

We start with the boundedness lemma:
\begin{lemma}
Let $a^\infty$ and $\mathcal E^{\infty}$ be given by \eqref{eq_ainfini}-\eqref{eq_einfini}. For $N \in \mathbb N$ sufficiently large, there holds:
$$
\|u^N\|_{H^1_0(\Omega)} \lesssim \sqrt{a^{\infty}} \, {\mathcal E}^{\infty}\,.
$$
\end{lemma}

\begin{proof}
We provide a proof of this lemma by applying the variational characterization of solutions to the Stokes problem \eqref{eq_varcar}:
we construct a divergence-free lifting of boundary conditions on $\partial B_i^N$ and $\partial \Omega.$ Any such candidate yields a bound  above on $\|\nabla u^N\|_{L^2(\mathcal F^N)}.$

\medskip

Given $N \in \mathbb N,$ we set:
    $$
    v^{N}(x) = \sum_{i=1}^N \nabla \times \left( \dfrac{\chi_{a^N}(x-h^N_i)}{2}\,  v^N_i \times (x- h^N_i)\right)  =: \sum_{i=1}^{N} v_{i}(x).
    $$  
Then, $v^N \in C^{\infty}_c(\mathbb R^3)$ is the curl of a smooth potential vector so that ${\rm div} \, v^N = 0\,.$ Because of assumption \eqref{eq_ass_dmin}, we have that, for $N$ sufficiently large:
$$
d_{min}^N > 4a^N.
$$
Since $\chi_{a^N}$ has support in $B_{\infty}(0, 2a^N)$ we have that $\text{Supp}(v_i) \subset B_{\infty}(h_i^N,2a^N)$ and the $(v_i)_{i=1,\ldots,N}$ have disjoint supports. Because $\chi_{a^N} = 1$ on $B(0,a^N) \subset B_{\infty}(0,a^N) $ we derive further that, for $i \in \{1,\ldots,N\}:$
\begin{eqnarray*}
v_i(x)  &=& 0\,, \quad \text{on $\partial \Omega \cup \bigcup_{j\neq i} B_j^N $}\,,\\
v_i(x) &=& 0\,, \quad  \text{on $B_j^N $ for $j \neq i$}\,,\\[4pt]
v_i(x) &=& \nabla \times \left( \dfrac{1}{2}\,  v^N_i \times (x- h^N_i)\right) = v_i^N\,,     \quad \text{ on $B_i^N$}.
\end{eqnarray*}
By combination, we obtain:
$$
\begin{array}{rcll}
v^N(x) &=& v_i^N\,, & \text{ on $\partial B_i^N$}\,, \quad \forall \, i=1,\ldots,N\,, \\[4pt]
v^N(x) &=& 0\,, & \text{ on $\partial \Omega$}\,.
\end{array}
$$
We have then by {Theorem \ref{thm_varcar}} that:
\begin{equation} \label{eq_rajout1}
\|\nabla u^{N}\|_{L^2(\mathcal F^N)}  \leq \|\nabla v^{N}\|_{L^2(\mathcal F^N)} = \left(\sum_{i=1}^N \|\nabla v_i\|^2_{L^2(\mathbb R^3)} \right)^{\frac 12}.
\end{equation}

 \medskip

For arbitrary $N \in \mathbb N$ and $i \in \{1,\ldots,N\},$ there holds:
 \begin{eqnarray*}
|\nabla v_{i}(x)| &\lesssim&   |\nabla \chi_{a^N}(x-h_i^N)| |v^N_i| + |\nabla^2 \chi_{a^{N}}(x-h_i^N)| |v^N_i||x-h_i^N|  \\
			&\lesssim& \dfrac{1}{a^N} \left( \left|\nabla\chi\left(\frac{x-h_i^N}{a^N}\right)\right| + \left|\nabla^2 \chi\left(\frac{x-h_i^N}{a^N}\right)\right| \right)|v^N_i|\,\\
 \end{eqnarray*}
 Consequently, by a standard scaling argument:
 $$
 \int_{\mathbb R^3} |\nabla v_i(x)|^2{\rm d}x \lesssim |a^N| \left( \int_{\mathbb R^3} |\nabla \chi(|y|)|^2 + |\nabla^2 \chi(|y|)|^2 {\rm d}y \right) |v^N_i|^2\,.
 $$
Since $a^N \leq a^{\infty}/N$ uniformly, we combine the previous computation into:
$$
 \|\nabla v^{N}\|_{L^2(\mathcal F^N)}^2  \lesssim   \dfrac{a^{\infty}}{N} \sum_{i=1}^{N} |v^N_i|^2\, \lesssim a^{\infty} | \mathcal E^{\infty}|^2.
 $$
Note that $\chi$ is fixed a priori so that all constants depending on $\chi$ may be considered as non-significant.  
\end{proof}
 
\bigskip

\section{Proof of Theorem \ref{thm_main} -- Computations for finite $N$} \label{sec_prfpart1}

From now on, we assume that $u^N$ converges weakly to $\bar{u}$ in $H^1_0(\Omega)$ (we do not relabel the subsequence for simplicity) and we fix a divergence-free $w \in C^{\infty}_c(\Omega).$
We recall that our aim is to compute the scalar product:
$$
\int_{\Omega} \nabla \bar{u} :  \nabla w 
$$
by applying that:
$$
\int_{\Omega} \nabla \bar{u} :  \nabla w 
 = \lim_{N \to \infty}  I^N \text{ with } I^N = \int_{\Omega} \nabla {u}^N : \nabla w\,, \quad \forall \, N \in \mathbb N\,.
$$

\medskip

We explain now the construction of the modified test-function $w^s.$ The integers $\delta \geq 4$ and $N \in \mathbb N^*$ are fixed in the remainder of this section. 
Some restrictions on the values of these parameters may be added in due course. Applying the construction in Appendix \ref{app_coveringlemma}, we obtain 
$(T^N_{\kappa})_{\kappa \in \mathbb Z^3}$ a covering of $\mathbb R^3$ with cubes of width $\lambda^N$ such that
denoting: 
$$
\mathcal Z_{\delta}^N := \left\{i \in \{1,\ldots,N\} \text{ s.t. } \text{dist} \left( h_i^N\,, \bigcup_{\kappa \in \mathbb Z^3} \partial T^N_{\kappa} \right) < \dfrac{\lambda^N}{\delta} \right\}\,,
$$
there holds:
\begin{equation} \label{eq_couloir}
\dfrac{1}{N}\sum_{i \in \mathcal Z_{\delta}^N} (1+|v_i^N|^2) \leq \dfrac{12}{\delta} \dfrac{1}{N} \sum_{i=1}^N (1+|v_i^N|^2)   \leq \dfrac{12 (1+|\mathcal E^{\infty}|^2)}{\delta}\,.
\end{equation}
Moreover, since $\lambda^N \to 0$ and $\text{Supp}(w) \Subset \Omega,$ for large $N,$ keeping only the indices $\mathcal K^N$ such that $T_{\kappa}^N$ intersects $\text{Supp}(w),$ we obtain a covering $(T_{\kappa}^N)_{\kappa \in \Kappa^N}$ 
of $\text{Supp}(w)$ such that all the cubes are included in $\Omega$ (see the appendix for more details). 
We  do not make precise the set of indices $\Kappa^N$. 
The only relevant property to our computations is that 
\begin{equation} \label{eq_propKappaN}
\card \Kappa^N \leq {|\Omega|}/|\lambda^{N}|^3 \,.
\end{equation}
This inequality is derived by remarking that the $T^N_{\kappa},$ $\kappa \in \Kappa^N,$ are disjoint 
cubes of volume $|\lambda^N|^3$ that are all included in $\Omega.$  
Associated to this covering, we introduce the following notations. For arbitrary $\kappa \in \Kappa^N,$
we set 
\begin{align*}
& \mathcal I_{\kappa}^N := \{ i \in \{1,\ldots, N \} \text{ s.t. } h_{i}^N \in T_{\kappa}^N\}\,, \quad M^N_{\kappa} := \card \mathcal{I}_{\kappa}^N \,, \\
&\mathcal I^N := \bigcup_{\kappa \in \Kappa^N} \mathcal I_{\kappa}^N\,.
\end{align*}
We note that \eqref{eq_ass_concentration} implies:
\begin{equation} \label{eq_propMkappa}
 M_{\kappa}^N \lesssim |\lambda^N|^3 N \,, \quad \forall \, \kappa \in \Kappa^{N}\,,  \quad \forall  \, N \in \mathbb N\,.
\end{equation}
and also that, since the $(T_{\kappa}^N)_{\kappa \in \Kappa^N}$ are disjoint, we have:
\begin{equation} \label{eq_propMkappa2}
 \sum_{\kappa \in \Kappa^N} M_{\kappa}^N \leq  N \,,  \quad \forall  \, N \in \mathbb N\,.
\end{equation}
In brief, the set of indices $\{1,\ldots,N\}$ contains the two important subsets: 
\begin{itemize}
\item the subset $\mathcal I^N$ contains all the indices that are "activated" in our computations,
\item the subset $\mathcal Z^N_{\delta}$ contains the indices that are close to boundaries of the partition.
%\item the subset $\{1,\ldots,N\} \setminus \mathcal I^N$ contains indices that are not activated.  
\end{itemize}
We emphasize that $\mathcal Z^N_{\delta}$ contains indices that can be in both $\mathcal I^N$ and its complement.
We also point out that, by assumption \eqref{eq_ass_dmin}-\eqref{eq_ass_lambdaN}, we have that $a^N$ decays like $1/N$ while $\lambda^N$ decays slowlier than $1/N^{1/3}$. 
In particular, for $N$ sufficiently large, given any $i \in \{1,\ldots,N\} \setminus \mathcal Z^{N}_{\delta}$ the center of $B_i^N = B(h_i^N,a^N)$ is contained in the $\lambda^N/\delta$-core of $T_{\kappa}^N$ which implies
that $B_i^N \subset T_{\kappa}^N$ for some $\kappa$ and  $B_i^N \cap T_{\kappa'}^N = \emptyset$ for $\kappa' \neq \kappa.$ On the other hand, if $i \in \mathcal Z^{N}_{\delta},$ since $a^N$ 
is much smaller than the width of the cubes $T_{\kappa}^N,$ we have that $B(h_i^N,2a^N)$ intersects $T_{\kappa}^N$ for at most $8$ distinct values of $\kappa \in \mathbb Z^3.$ This properties will be recalled in due course below.
%The reader should bear in mind these remarks that will be repeated in the computations below.

\medskip

We construct then $w^{s}$ piecewisely on the covering of $\text{Supp}(w).$ Given $\kappa \in \mathcal \Kappa^N,$ we set:
\begin{equation} \label{eq_wsN}
w^s_{\kappa}(x) = \sum_{i \in \mathcal I_{\kappa}^N \setminus \mathcal Z_{\delta}^N} U^{a^N}[w(h_i^N)](x-h_i^N)\,, \quad \forall \, x \in \mathbb R^3\,,
\end{equation}
and 
$$
w^{s} = \sum_{\kappa \in \Kappa^N} w^s_{\kappa}  \mathbf{1}_{T_{\kappa}^N}.
$$
We  note that $w^{s} \notin H^1_0(\mathcal F^N)$ because of jumps at interfaces $\partial T_{\kappa}^N.$ 
 It will be sufficient for our purpose that $w^s \in H^1(\mathring{T}_{\kappa}^N)$ 
for arbitrary $\kappa \in \Kappa^{N}.$ In a cube $\mathring{T}_{\kappa}^N$ the test function $w^s$ is thus a combination of stokeslets centered in the $h_i^N$ that are contained
in the cell. We delete from this combination the centers that are too close to $\partial T_{\kappa}^N$ (namely $\lambda^N/\delta$-close to $\partial T_{\kappa}^N$).

\medskip

We proceed by proving that we make a small error by replacing $w$ with $w^s$ in $I^N.$ We emphasize that, in the next statement as in the following ones, we use the symbol $\lesssim$ for inequalities in which, on the right-hand side, we keep only the terms depending on the parameters 
$\delta$ and $N$ (unless other quantities are required to make the computations more clear): 
\begin{proposition} \label{prop_decomposition}
There exists a constant $C_{\delta}$ such that, for $N$ sufficiently large,  there holds:
\begin{multline} \label{eq_decomposition}
\left| 
\int_{\Omega} \nabla u^N : \nabla w - \sum_{\kappa \in \Kappa^N} \int_{T_{\kappa}^N} \nabla u^N : \nabla w^s  
\right| \\
\lesssim \left( \frac 1{\delta^{\frac 13}}   + \dfrac{1}{|Nd_{min}^N|^{\frac 13}} + C_{\delta}  \left(\lambda^N + |N^{1/6}\lambda^N|^2 \right)  \right)\,.
\end{multline}
\end{proposition}
\begin{proof}
We split the proof in several steps by introducing different intermediate 
test-functions.

\medskip

\paragraph{\bf \em First step: Construction of auxiliary test-functions} Let fix $\kappa \in \Kappa^N,$ and consider the Stokes problem 
on $\mathring{T}_{\kappa}^N \setminus \bigcup_{i \in \mathcal I^N_{\kappa}\setminus \mathcal Z_{\delta}^N} \overline{B_i^N}$ with boundary conditions:
\begin{equation} \label{cab_stokesenplus}
\left\{
\begin{array}{rcll}
u(x) &=& w(x)  \,, & \text{ on $\partial B_{i}^N$ for $i \in  \mathcal I^N_{\kappa}\setminus \mathcal Z_{\delta}^N$}\,, \\[6pt]
u(x) &=& 0 \,, & \text{ on $\partial T_{\kappa}^N$}\,.
\end{array}
\right.
\end{equation}
We note that this problem enters the framework of Section \ref{sec_truncationprocess}. Indeed, let denote:
$$
d_{m}^{\kappa} := \min_{i \neq j \in  \mathcal I^N_{\kappa}\setminus \mathcal Z_{\delta}^N}  |h_i^N - h_j^N|, 
$$ 
and remark that, since we deleted the indices of $\mathcal Z_{\delta}^N,$ we have:
$$
\min_{i \in \mathcal I^N_{\kappa}\setminus \mathcal Z_{\delta}^N} {\rm dist}(h_i^N,\partial T_{\kappa}^N) \geq \dfrac{\lambda^N}{\delta}. 
$$ 
Moreover, since $d_{m}^{\kappa}$ and $\lambda^N$ converge to $0$ much slowlier than the radius $a^N$ of the particles, there holds for $N$ sufficiently large:
\begin{equation} \label{eq_boundbelowdkappam}
\min \left( d_{m}^{\kappa}, \frac{\lambda^N}{\delta} \right) \geq 4a^N\,.
\end{equation}
whatever the value of $\kappa \in \Kappa^N.$

\medskip

So, for $N$ sufficiently large, assumption \eqref{eq_hypdm} is satisfied and   the arguments developed in {Section \ref{sec_truncationprocess}} entail that there exists a unique generalized
solution to the Stokes problem on $\mathring{T}_{\kappa}^N \setminus \bigcup_{i \in \mathcal I^N_{\kappa}\setminus \mathcal Z_{\delta}^N} \overline{B_i^N}$
with boundary condition \eqref{cab_stokesenplus}. We denote this solution by $\bar{w}_{\kappa}$.  We keep notation $\bar{w}_{\kappa}$ to denote its extension to $\Omega$ (by $w$ on the holes and
by $0$ outside $\mathring{T}_{\kappa}^N$). As $\mathring{T}_{\kappa}^N \subset \Omega,$  we obtain a divergence-free $\bar{w}_{\kappa} \in H^1_0(\Omega).$ We then add the $\bar{w}_{\kappa}$ into:
$$
\bar{w} = \sum_{\kappa \in \Kappa^N} \bar{w}_{\kappa}\,.
$$  
This vector-field satisfies:
\begin{itemize}
\item $\bar{w} \in H^1_0(\Omega),$
\item ${\rm div} \, \bar{w} = 0$ on $\Omega,$
\item $\bar{w} = w$ on $B_i^N$ for all $i \in \{1,\ldots,N\} \setminus \mathcal Z_{\delta}^N.$
\end{itemize}
The only statement that needs further explanation is the last one. By construction, we have clearly that $\bar{w} = w$ on $B_{i}^N$ for all $i \in \mathcal I^N \setminus \mathcal Z_{\delta}^N.$
When $i \in \{1,\ldots,N\} \setminus ( \mathcal I^N \cup \mathcal Z^N_{\delta})$ we have on the one hand that 
$h_i^N\notin \mathcal I^N$ so that the only index $\kappa$ satisfying $h_i^N \in T_{\kappa}^N$ is not in  $\Kappa^N.$ This entails that $w = 0$ on $B_i^N.$ On the other hand we have that $i \notin \mathcal Z_{\delta}^N$ so that $h_i^N$ is in the $\lambda^N/\delta$ core of this $T_{\kappa}^N$ and $B(h_i^N,a^N) \subset  T_{\kappa}^N.$ This entails that $\bar{w} = 0$ on $B_i^N.$ Finally, we obtain that $\bar{w} = w = 0$ on $B_i^N.$ 

\medskip

We correct now the value of $\bar{w}$ on the $B_{i}^N$ when $i \in \mathcal Z_{\delta}^N$ in order that it fits the same boundary
conditions as $w$ on $\mathcal F^N.$ We set:
\begin{eqnarray*}
\tilde{w} &=& \sum_{i \in \mathcal Z_{\delta}^N} \left[ \chi_{a^N}(\cdot -h_i^N) w -  \mathfrak B_{h_i^N,a^N,2a^N} [x \mapsto w(x) \cdot \nabla \chi_{a^N}(x-h_i^N)] \right]  \\
		&& \; + \prod_{i \in \mathcal Z_{\delta}^N} (1- \chi_{a^N}(\cdot -h_i^N)) \bar{w} + \sum_{i \in \mathcal Z_{\delta}^N} \mathfrak B_{h_i^N,a^N,2a^N} [x \mapsto \bar{w}(x) \cdot \nabla \chi_{a^N}(x-h_i^N)]\,.
\end{eqnarray*}
One may interpret the construction of $\tilde{w}$ as follows. The sum on the first line creates a divergence-free lifiting of the boundary conditions prescribed by $w$ on the $\partial B_i^N$   for $i \in \mathcal Z_{\delta}^N$.
On the second line is a divergence-free truncation of $\bar{w}$ that creates a vector-field vanishing on $\cup_{i \in \mathcal Z_{\delta}^N} B_i^N.$  We remark that this vector-field is well-defined since, by straightforward integration by parts, we have:
$$
\int_{A(h_i^N,a^N,2a^N)} \bar{w}(x) \cdot \nabla \chi^N(x-h_i^N) {\rm d}x = \int_{A(h_i^N,a^N,2a^N)}  w(x) \cdot \nabla \chi^N(x-h_i^N) {\rm d}x = 0\,, \quad \forall \, i \in \mathcal Z_{\delta}^N.
$$
Hence, we may apply the Bogovskii operator which lifts the divergence term in the brackets with a vector-field vanishing on the boundaries of $A(h_i^N,a^N,2a^N)$
that we extend by $0$ on $\mathbb R^3 \setminus A(h_i^N,a^N,2a^N).$

\medskip

For $N$ sufficiently large, the family of balls $(B_{\infty}(h_{i}^N,2a^N))_{i=1,\ldots,N}$ are disjoint and included in $\Omega.$
Consequently, direct computations show that ${\rm div}\, \tilde{w} = 0$ on $\Omega,$ and that the truncations that we perform in $\tilde{w}$ do not perturb the value of $\bar{w}$ neither on the $B_i^N$ for $i \in \{1,\ldots,N\} \setminus \mathcal Z_{\delta}^N$ nor on $\partial \Omega.$ This remark entails that
\begin{itemize}
\item for  $i \in \{1,\ldots,N\} \setminus \mathcal Z_{\delta}^N$:  
$$
\tilde{w}(x) = \bar{w}(x) = w(x)\,, \quad  \text{on $B_{i}^N$}\,,
$$
\item for $i \in  \mathcal Z_{\delta}^N$: 
$$
\tilde{w}(x)= \chi_{a^N}(x -h_i^N) w(x)  = w(x)\,, \quad  \text{on $B_{i}^N$}\,,
$$ 
\item $\tilde{w}(x) = 0\,,$ on $\partial \Omega.$
\end{itemize} 

Consequently, by restriction, there holds that $w - \tilde{w} \in H^1_0(\mathcal F^N)$ is divergence-free. 
As $u^N$ is a generalized solution to a Stokes problem on $\mathcal F^N$ we have thus:
$$
\int_{\mathcal F^N} \nabla u^N : \nabla (w- \tilde{w}) = 0.
$$
%Remarking that $\nabla u^N$ vanishes on the complement of $\mathcal F^N$ in $\Omega,$ this entails: 
%$$
%\int_{\Omega} \nabla u^N : \nabla w = \sum_{\kappa \in \Kappa^N} \int_{T^N_{\kappa}} \nabla u^N : \nabla \tilde{w}. 
%$$
We rewrite this identity as follows:
\begin{equation} \label{eq_cpreskesa}
\int_{\Omega} \nabla u^N : \nabla w = \sum_{\kappa \in \Kappa^N} \int_{T^N_{\kappa}} \nabla u^N : \nabla w^s - E_1 - E_2\,,
\end{equation}
with :
\begin{eqnarray*}
E_1 &=& \sum_{\kappa \in \Kappa^N} \int_{T^N_{\kappa}} \nabla u^N : \nabla (w^s_{\kappa} - \bar{w}_{\kappa})  \,, \\
E_2 &=&  \int_{\Omega} \nabla u^N : \nabla (\bar{w} - \tilde{w})\,.  
\end{eqnarray*}

\paragraph{\bf \em Second step: Control of error term $E_1$.}
 For arbitrary $\kappa \in \Kappa^N,$ we apply {Proposition \ref{prop_truncationprocess}} to $\bar{w}_{\kappa}$
 and its corresponding combination of stokeslets (namely, the restriction $w^s_{\kappa}$ of $w^s$ to $\mathring{T}^N_{\kappa}$). 
 We have thus (keeping only the largest terms):
 \begin{eqnarray*}
 \|\nabla (w^s_{\kappa} - \bar{w}_{\kappa})\|_{L^2(T^N_{\kappa})} & \lesssim  &
   \sqrt{M_{\kappa}^N|a^N|^{3}}\left[  1 +  \dfrac{|M_{\kappa}^{N}|^{2/3}}{d^{\kappa}_{m}} +  \dfrac{a^N |M_{\kappa}^N|^{1/3}}{|d^{\kappa}_{m}|^2}    \right] +C_{\delta}M_{\kappa}^Na^N \left[ \lambda^N  + \frac 1{\lambda^N}\right]^{\frac 12}\\
  & \lesssim & \sqrt{\frac{M_{\kappa}^N}{N}} \left[ \frac{1}{N} + \dfrac{|M_{\kappa}^{N}|^{2/3}}{Nd^{\kappa}_{m}} + C_{\delta} \sqrt{\dfrac{M_{\kappa}^N}{N{\lambda^N}}} \right]
 \end{eqnarray*}
We applied here that $\card ( \mathcal I_{\kappa}^N \setminus \mathcal Z_{\delta}^N) \leq \card \mathcal I_{\kappa}^N = M^N_{\kappa} \geq 1$ and that $Nd_{m}^{\kappa} \geq N d_{min}^N >> 1$ which entail: 
$$
\dfrac{|a^N|^2|M_{\kappa}^N|^{1/3}}{|d^{\kappa}_{m}|^2} \lesssim \dfrac{1}{Nd_{m}^{\kappa}|M_{\kappa}^N|^{1/3}}\dfrac{|M_{\kappa}^{N}|^{2/3}}{Nd^{\kappa}_{m}}  << \dfrac{|M_{\kappa}^{N}|^{2/3}}{Nd^{\kappa}_{m}} .
$$
Adding \eqref{eq_ainfini} and applying a standard Cauchy-Schwarz inequality together with \eqref{eq_propMkappa}-\eqref{eq_propMkappa2} yields:
\begin{eqnarray*}
|E_1 | &\lesssim&  \sum_{\kappa \in \Kappa^N } \|\nabla u^N\|_{L^2(T^N_{\kappa})} \sqrt{\frac{M_{\kappa}^N}{N}} \left[ \frac{1}{N} + \dfrac{|M_{\kappa}^{N}|^{2/3}}{Nd^{\kappa}_{m}} + C_{\delta} \sqrt{\dfrac{M_{\kappa}^N}{N{\lambda^N}}} \right]\,,\\
	& \lesssim &  \left( \sum_{\kappa \in \Kappa^N} \|\nabla u^N\|^2_{L^2(T^N_{\kappa})}\right)^{\frac 12} \left( \dfrac{1}{{N}} + \dfrac{|\lambda^N|^2}{N^{1/3}d_{min}^N} + C_{\delta} |\lambda^N|\right)\,.
\end{eqnarray*}
Here, we note again that, by construction, the $T^N_{\kappa}$ are disjoint and included in $\Omega$ so that
$$ 
\left( \sum_{\kappa \in \Kappa^N} \|\nabla u^N\|^2_{L^2(T^N_{\kappa})}\right)^{\frac 12} \leq \|\nabla u^N\|_{L^2(\Omega)} .
$$
Applying the uniform bound for $u^N$ in $H^1_0(\Omega)$  and  introducing \eqref{eq_boundbelowdkappam} and \eqref{eq_ass_lambdaN}, we conclude then that:
\begin{equation} \label{eq_boundE1}
|E_1 | \lesssim  \left( \dfrac{1}{|Nd_{min}^N|^{1/3}} + C_{\delta} |\lambda^N|\right)
\,.
\end{equation}

\paragraph{\bf \em Third step: Control of error term $E_2$.} 
As for the second term, we replace $\tilde{w}$ by its explicit construction. We remark that, because the supports of the $(\chi_{a^N}(\cdot - h_i^N))_{i \in \{1,\ldots,N\}}$ are disjoint (as $d_{min}^N > 4a^N$), we have:
$$
1 - \prod_{i \in \mathcal Z_{\delta}^N} (1- \chi_{a^N}(x-h_i^N)) =   \sum_{i \in \mathcal Z_{\delta}^N} \chi_{a^N}(x-h_i^N)\,, \quad \forall \, x \in \Omega.
$$
Consequently, we  split:
\begin{eqnarray*}
\bar{w} - \tilde{w} &=& \sum_{i \in  \mathcal Z_{\delta}^N} \left[ \chi_{a^N}(\cdot -h_i^N) \bar{w} -  \mathfrak B_{h_i^N,a^N,2a^N} [x \mapsto \bar{w}(x) \cdot \nabla \chi_{a^N}(x-h_i^N)] \right]  \\
		&& \; -  \sum_{i \in \mathcal Z_{\delta}^N} \left[ \chi_{a^N}(\cdot - h_i^N) {w} -   \mathfrak B_{h_i^N,a^N,2a^N} [x \mapsto {w}(x) \cdot \nabla \chi_{a^N}(x-h_i^N)]\right]\,.
\end{eqnarray*}
We note in particular that:
$$
\text{Supp} ( \bar{w} - \tilde{w}) \subset \bigcup_{i \in \mathcal Z^N_{\delta}} B_{\infty}(h_i^N,2a^N).
$$
Since the balls appearing on the right-hand side of this identity are disjoint, we have:
\begin{eqnarray*}
|E_2| &\leq&  \sum_{i \in \mathcal Z_{\delta}^N} \int_{B_{\infty}(h_i^N,2a^N)} |\nabla u^N|  | \nabla  ( \bar{w} - \tilde{w})|  \,, \\
	& \leq & \left( \sum_{i \in \mathcal Z_{\delta}^N} \|\nabla u^N\|^2_{L^2(B_{\infty}(h_i^N,2a^N))}\right)^{\frac 12}\left( \sum_{i\in \mathcal Z_{\delta}^N} \| \nabla (\bar{w} - \tilde{w})\|^2_{L^2(B_{\infty}(h_i^N,2a^N)))}\right)^{\frac 12} \,,\\
	&\lesssim & \left( \sum_{i\in \mathcal Z_{\delta}^N} \| \nabla ( \bar{w} -  \tilde{w})\|^2_{L^2(B_{\infty}(h_i^N,2a^N)))}\right)^{\frac 12}\,.
\end{eqnarray*}
Given $i \in \mathcal Z^N_{\delta},$ direct computations and application of {Lemma \ref{lem_div}} to the Bogovskii operator $\mathfrak B_{h_i^N,a^N,2a^N}$ entail that (recalling \eqref{eq_ainfini} to bound $a^N$):
$$
\| \nabla\left(  \bar{w} - \tilde{w}\right) \|^2_{L^2(B_{\infty}(h_i^N,2a^N))}  \lesssim \frac 1N \|w\|^2_{W^{1,\infty}} + N^2 \|\bar{w}\|^2_{L^2(B_{\infty}(h_i^N,2a^N))}  + \|\nabla \bar{w}\|^2_{L^2(B_{\infty}(h_i^N,2a^N))}\,.
$$
Consequently, we have the bound:
\begin{equation} \label{eq_tobound}
\sum_{i\in \mathcal Z_{\delta}^N} \| \nabla (\bar{w} - \tilde{w}) \|^2_{L^2(B_{\infty}(h_i^N,2a^N)))} \lesssim E_{2,\infty} + N^2 E_{2,0} + E_{2,1}
\end{equation}
with :
$$
E_{2,\infty} := \sum_{i\in \mathcal Z_{\delta}^N}\frac 1N \|w\|^2_{W^{1,\infty}} \quad  E_{2,0} := \sum_{i\in \mathcal Z_{\delta}^N} \|\bar{w}\|^2_{L^2(B_{\infty}(h_i^N,2a^N))} \quad 
E_{2,1} := \sum_{i\in  \mathcal Z_{\delta}^N}\|\nabla \bar{w}\|^2_{L^2(B_{\infty}(h_i^N,2a^N))}.
$$
We bound now these three terms independantly.

\medskip

First, we recall that, by choice of the covering (see \eqref{eq_couloir}), we have:
\begin{equation} \label{eq_propdelta}
\sharp \mathcal Z_{\delta}^N  \lesssim \dfrac{N}{\delta}.
\end{equation}
Consequently, there holds:
\begin{equation}\label{eq_E2infini}
E_{2,\infty} = \sum_{i\in \mathcal Z_{\delta}^N} \dfrac{1}{N} \|w\|^2_{W^{1,\infty}} \lesssim \dfrac{ 1 }{\delta}\,.
\end{equation}

\medskip

Second, we remark that $(T_{\kappa}^N)_{\kappa \in \Kappa^N}$ is also a covering of Supp$(\bar{w}),$ so that we have:
$$
E_{2,0} =  \sum_{i\in \mathcal Z_{\delta}^N}  \|\bar{w}\|^2_{L^2(B_{\infty}(h_i^N,2a^N))} = \sum_{i \in \mathcal Z_{\delta}^N} \sum_{\kappa \in \Kappa^N} \|\bar{w}\|^2_{L^2(B_{\infty}(h_i^N,2a^N) \cap T_{\kappa}^N)} \, .
$$
Given $\kappa \in \Kappa^N,$ we apply now {Proposition \ref{prop_comparestks2}} to approximate $\bar{w}_{\kappa}$ by a combination of stokeslet $\bar{w}_{\kappa}^s.$
We remark here that because of \eqref{eq_ass_lambdaN} we have that:
$$
\dfrac{|M_{\kappa}^N|^{2/3} a^N}{d_{min}^N} \lesssim \dfrac{|\lambda^N|^2}{N^{1/3} d_{min}^N} \lesssim \dfrac{1}{|Nd_{min}^N|^{1/3}} << 1.
$$
Consequently, the assumptions of {Proposition \ref{prop_comparestks2}} are satisfied for all $\kappa \in \Kappa^N$ for $N$ sufficiently large.
We have then:
$$ 
E_{2,0} \lesssim  \sum_{\kappa \in \Kappa^N}  \sum_{i \in \mathcal Z_{\delta}^N}  \|\bar{w}_{\kappa} - \bar{w}^s_{\kappa} \|^2_{L^2(B_{\infty}(h_i^N,2a^N) \cap T_{\kappa}^N)} + 
		\sum_{i \in \mathcal Z_{\delta}^N} \sum_{\kappa \in \Kappa^N} \|\bar{w}_{\kappa}^s\|^2_{L^2(B_{\infty}(h_i^N,2a^N) \cap T_{\kappa}^N)}\,.
$$

We compute the terms involving $\bar{w}^s_{\kappa}$ by using the explicit formula \eqref{eq_wsN} and the expansion of stokeslet \eqref{eq_Stokesexpansion}. 
To this end, we remind that $B_{\infty}(h_i^N,2a^N) \cap T_{\kappa}^N \neq \emptyset$ implies that $h_i^N$ is in the $2a^N$-neighborhood of $T_{\kappa}^N.$ 
As the width of a $T_{\kappa}^N$ is much larger than $2a^N$ this implies that this property is satisfied by at most $8$ cubes. We have thus:
$$
\sum_{\kappa \in \Kappa^N} \|\bar{w}_{\kappa}^s\|^2_{L^2(B_{\infty}(h_i^N,2a^N) \cap T_{\kappa}^N)} \leq 8 \sup_{\kappa \in \Kappa^N} \|\bar{w}_{\kappa}^s\|^2_{L^2(B_{\infty}(h_i^N,2a^N))}\,, \quad \forall \, i \in \mathcal Z_\delta^N\,.
$$
Given $i\in \mathcal Z_{\delta}^N,$  $\kappa \in \Kappa^N$ and $j \in \mathcal I^N_{\kappa} \setminus \mathcal Z_{\delta}^N$ the distance between $h_{j}^N$
and $B_{\infty}(h_i^N,2a^N)$ is larger than $d^N_{min}/2.$ Recalling that there are at most $M^N_{\kappa}$ indices in $\mathcal I_{\kappa}^N \setminus \mathcal Z^{N}_{\delta}$ 
-- and that, by the first item of proposition \ref{prop_comparestks2}, the coefficients of the combination of stokeslet $\bar{w}_{\kappa}^s$ are close to the $w(h_i)$ -- we derive the bound:  
$$
|\bar{w}^{s}_{\kappa}(x)| \lesssim \dfrac{M^N_{\kappa}a^N}{d^N_{min}} \,, \quad \forall \, x \in B_{\infty}(h_i^N,2a^N)\,.
$$ 
Consequently, there holds
\begin{eqnarray*}
\| \bar{w}^s_{\kappa}\|^2_{L^2(B_{\infty}(h_i^N,2a^N))} &\lesssim&  \dfrac{|M^N_{\kappa}|^2|a^N|^5}{|d_{min}^N|^2}  \lesssim  \dfrac{|\lambda^N|^6}{N^3 |d_{min}^N|^2}  \lesssim \dfrac{1}{N^{3}} \,,				
\end{eqnarray*}
where we applied  \eqref{eq_propMkappa} in the before last inequality and assumption \eqref{eq_ass_lambdaN} for the last one. This yields, due to our choice of covering:
\begin{eqnarray}
\notag \sum_{i \in \mathcal Z_{\delta}^N} \sum_{\kappa \in \Kappa^N} \|\bar{w}_{\kappa}^s\|^2_{L^2(B_{\infty}(h_i^N,2a^N) \cap T_{\kappa}^N)} &\lesssim&  \dfrac{1}{N^2}\sum_{i \in \mathcal Z_{\delta}^N}\dfrac{1}{N} \\
																									& \lesssim &\dfrac{ 1}{\delta N^2 } \label{eq_E201}	\,.
\end{eqnarray}

For the remainder terms, we proceed by applying a H\"older inequality. For arbitrary $\kappa \in \Kappa^N,$ we introduce   $\mathcal Z_{\delta,\kappa}^N$ the set of indices $i \in \mathcal Z_{\delta}^N$ such that $B_{\infty}(h_i^N,2a^N)\cap T_{\kappa}^N \neq \emptyset,$ 
we infer by a H\"older inequality that:
\begin{eqnarray*}
\sum_{i \in \mathcal Z_{\delta,\kappa}^N} \| \bar{w}_{\kappa} - w^s_{\kappa}\|^2_{L^2(B_{\infty}(h_i^N,2a^N) \cap T_{\kappa}^N)} &\lesssim&  \dfrac{|\sharp \mathcal Z_{\delta,\kappa}^N|^{\frac 23}}{N^2}  \| (\bar{w}_{\kappa} - w^s_{\kappa})\|^2_{L^6(T_{\kappa}^N)} \, , 
\end{eqnarray*}
and, by applying again a H\"older inequality:
\begin{multline*}
\sum_{\kappa \in \Kappa^N}  \sum_{i \in \mathcal Z_{\delta}^N}  \|\bar{w}_{\kappa} - w^s_{\kappa} \|^2_{L^2(B_{\infty}(h_i^N,2a^N) \cap T_{\kappa}^N)}   
 \lesssim   \dfrac{1}{N^2}\left[ \sum_{\kappa \in \Kappa^N} \sharp \mathcal Z_{\delta,\kappa}^N\right]^{\frac 23}  \left[ \sum_{\kappa \in \Kappa^N} \| (\bar{w}_{\kappa} - w^s_{\kappa})\|^6_{L^6(T_{\kappa}^N)}\right]^{1/3}.
\end{multline*}
At this point, we remark again that $\mathcal Z^N_{\delta,\kappa} \subset \mathcal Z^N_{\delta}$ and that one index $i\in \mathcal Z^N_{\delta}$ belongs to at most $8$ distinct sets $\mathcal Z_{\delta,\kappa}^N$ so that:
$$
\left[ \sum_{\kappa \in \Kappa^N} \sharp \mathcal Z_{\delta,\kappa}^N\right]^{\frac 23} \lesssim \left[ \sharp \mathcal Z_{\delta}^N \right]^{\frac 23} \lesssim \left|\dfrac{N}{\delta} \right|^{\frac 23} \,.
$$
We  also remark that Proposition \ref{prop_comparestks2},  combined with \eqref{eq_ainfini}, \eqref{eq_ass_lambdaN}, \eqref{eq_propMkappa} and \eqref{eq_boundbelowdkappam}
entails:
\begin{eqnarray*}
 \| (\bar{w}_{\kappa} - \bar{w}^s_{\kappa})\|^6_{L^6(T_{\kappa}^N)}  &\lesssim&  \left(\dfrac{M_{\kappa}^N}{N^3} +   \dfrac{|M_{\kappa}^N|^{5/3}}{N^5|d_{min}^N|^4}+   C_{\delta}\dfrac{|M_{\kappa}^N|^2}{N^2\lambda^N} \right)^3 
 %{\frac{M_{\kappa}^N}{N}} \left[ \frac{1}{N} + \dfrac{|M_{\kappa}^{N}|^{2/3}}{Nd^{\kappa}_{m}} + C_{\delta} \sqrt{\dfrac{M_{\kappa}^N}{N{\lambda^N}}} \right]^2 \,, \\
% 
\end{eqnarray*}
 and, recalling \eqref{eq_propMkappa2}: 
 $$
 \left(\sum_{\kappa \in \Kappa^N} \| (\bar{w}_{\kappa} - w^s_{\kappa})\|^6_{L^6(T_{\kappa}^N)} \right)^{1/3} \lesssim  \left(\dfrac{|\lambda^N|^2}{N^2} +   \dfrac{|\lambda^N|^4}{N^{10/3}|d_{min}^N|^4}+   C_{\delta}|\lambda^N|^{4}   \right) 
 $$
This entails:
\begin{multline*}
\sum_{\kappa \in \Kappa^N} \sum_{i \in \mathcal Z_{\delta}^N} \| (\bar{w}_{\kappa} - w^s_{\kappa})\|^2_{L^2(B_{\infty}(h_i^N,2/N) \cap T_{\kappa}^N)} \\
\lesssim \dfrac{1}{\delta^{2/3}N^2}
\left(\dfrac{|\lambda^N|^2}{N^{4/3}} +   \dfrac{|\lambda^N|^4}{|N^{2/3}d_{min}^N|^4}+   C_{\delta}(N^{2/3} |\lambda^N|^4)   \right) 
\end{multline*}
Finally, recalling  \eqref{eq_ass_lambdaN} to control the second term in parenthesis and combining with \eqref{eq_E201} we obtain:
\begin{equation} \label{eq_E20}
E_{2,0} \lesssim  \dfrac{1}{N^2} \left( \dfrac{1}{\delta^{\frac 2{3}}}+  C_{\delta}|N^{1/6}\lambda^N|^4  \right) .
\end{equation}
By decomposing again $\nabla \bar{w}$ into $\nabla (\bar{w} - \bar{w}^s_{\kappa}) +  \nabla \bar{w}^s_{\kappa}$, applying Proposition \ref{prop_comparestks2} $ii)$ to estimate the first term and the expansion 
\eqref{eq_Stokesexpansion} with Proposition \ref{prop_comparestks2} $i)$ for the second one, we obtain that
$$
E_{2,1} :=  \sum_{i\in \mathcal Z_{\delta}^N} \|\nabla \bar{w}\|^2_{L^2(B_{\infty}(h_i^N,2a^N))}\,.
$$
satisfies the similar bound:
\begin{equation} \label{eq_E21}
E_{2,1} \lesssim  \left( \dfrac{1}{\delta} +  \dfrac{1}{|Nd_{min}^N|^{10/3}} + C_{\delta} |\lambda^N|^2  \right) \,.
\end{equation}
Eventually, gathering \eqref{eq_E2infini}, \eqref{eq_E20} and \eqref{eq_E21} in \eqref{eq_tobound} entails:
\begin{equation} \label{eq_boundE2}
E_{2} \lesssim  \left(  \dfrac{1}{\delta^{\frac 23}} +   \dfrac{1}{|Nd_{min}^N|^{{10}/3}} + C_{\delta} \left(|\lambda^N|^2 +|N^{1/6}\lambda^N|^4  \right)  \right)^{\frac 12} .
\end{equation}
We complete the proof by combining \eqref{eq_boundE1} and \eqref{eq_boundE2} in \eqref{eq_cpreskesa}.
\end{proof}

\section{Proof of Theorem \ref{thm_main} -- Asymptotics $N \to \infty$} \label{sec_pfpart2}
In this section, we end the proof of Theorem \ref{thm_main} keeping the notations
introduced in the previous section. Under assumption \eqref{eq_ass_dmin}-\eqref{eq_ass_lambdaN}, a straightforward corollary of Proposition \ref{prop_decomposition}  reads:
\begin{corollary}
For arbitrary $\delta \geq 4,$ there holds:
$$
\limsup_{N\to \infty} \left|\int_{\Omega} \nabla u^N : \nabla w - \sum_{\kappa \in  \Kappa^N }\int_{T_{\kappa}^N} \nabla u^N : \nabla w^s \right| \lesssim \dfrac{1}{{\delta}^{\frac 13}}
$$
\end{corollary}
So in this section, we prove the following proposition:
\begin{proposition} \label{prop_ItildeN}
For arbitrary  $\delta \geq 4,$ there holds:
$$
\limsup_{N\to \infty} \left|\sum_{\kappa \in \mathcal K^N }\int_{T_{\kappa}^N} \nabla u^N : \nabla w^s - 6\pi a \int_{\Omega} (j - \rho \bar{u}) \cdot w \right| \lesssim \dfrac{1}{\sqrt{\delta}}  \,.
$$ 
\end{proposition}

This will end the proof of {Theorem \ref{thm_main}}. Indeed, combining the above corollary and proposition, we obtain that 
there exists $K$ which does not depend on $\delta$ such that, for arbitrary  $\delta \geq 4$:
$$
\limsup_{N\to \infty} \left|\int_{\Omega} \nabla u^N : \nabla w -6\pi a \int_{\Omega} (j - \rho \bar{u}) \cdot w \right| \leq \dfrac{K}{\sqrt{\delta}}\,.
$$ 
As 
$$
\lim_{N \to \infty} \int_{\Omega} \nabla u^N : \nabla w = \int_{\Omega} \nabla \bar{u} : \nabla w\,,
$$
and $\delta$ can be made arbitrary large, this entails that 
$$
\int_{\Omega} \nabla \bar{u} : \nabla w = 6\pi a \int_{\Omega} (j - \rho \bar{u}) \cdot w\,,
$$
and we obtain that $\bar{u}$ satisfies (B3).

\medskip

We give now a proof of Proposition \ref{prop_ItildeN}. For $\delta \geq 4$ and $N$ sufficiently large, we denote:
$$
\tilde{I}^N = \sum_{\kappa \in \Kappa^N} \int_{T_{\kappa}^N} \nabla u^N : \nabla w^s =  \sum_{\kappa \in \Kappa^N} \int_{T_{\kappa}^N} \nabla u^N : \nabla w^s_{\kappa}.
$$
Let fix $\kappa \in \Kappa^N$ and denote
 $$
 \tilde{I}^N_{\kappa} := \int_{T_{\kappa}^N} \nabla u^N : \nabla w^{s}_{\kappa}\,.
 $$
At first, we give a simpler expression for this integral. By definition, we have that:
$$
w^{s}_{\kappa} (x) = \sum_{i \in \mathcal I_{\kappa}^N \setminus \mathcal Z_{\delta}^N} U^{a^N}[w(h_i^N)](x-h_i^N)\,, \quad \forall \, x \in \mathbb R^3\,,
$$
so that, introducing the associated pressures $ x \mapsto P^{a^N}[w(h_i^N)](x-h_i^N),$ we obtain

\begin{eqnarray*}
\tilde{I}_{\kappa}^N &=& \int_{T_{\kappa}^N} \nabla u^N : \nabla w^{s}_{\kappa}\,, \\
							&=& \sum_{i \in \mathcal I^{N}_{\kappa} \setminus \mathcal Z_{\delta}^N} \int_{T_{\kappa}^N \setminus B_i^N} \nabla u^N(x) : [ \nabla  U^{a^N}[w(h_i^N)](x-h_i^N) - P^{a^N}[w(h_i^N)](x-h_i^N)\mathbb I_3] {\rm d}x\,, 
\end{eqnarray*}
Since $u^N$ is divergence-free, we integrate by parts. This yields:
\begin{eqnarray*}
	\tilde{I}_{\kappa}^N						&=&\sum_{i \in \mathcal I^{N}_{\kappa} \setminus \mathcal Z_{\delta}^N}   \Bigg(  \int_{\partial T_{\kappa}^N} \big\{ \partial_n  U^{a^N}[w(h_i^N)](x-h_i^N) - P^{a^N}[w(h_i^N)](x-h_i^N)n\big\} \cdot u^N(x) {\rm d}\sigma  \\
						&& \; -							\int_{\partial B_i^N} \big\{ \partial_n  U^{a^N}[w(h_i^N)](x-h_i^N) - P^{a^N}[w(h_i^N)](x-h_i^N)n\big\} \cdot v_i^N {\rm d}\sigma \Bigg) \,, \\
						&=& 		\sum_{i \in \mathcal I^{N}_{\kappa} \setminus \mathcal Z_{\delta}^N}  	I_{i,ext}^N - I_{i,int}^N	\,,			
\end{eqnarray*}
where we denoted:
\begin{eqnarray*}
 I_{i,int}^N &=& \int_{\partial B_i^N} \big\{ \partial_n  U^{a^N}[w(h_i^N)](x-h_i^N) - P^{a^N}[w(h_i^N)](x-h_i^N)n\big\} \cdot v_i^N {\rm d}\sigma\,,  \\
I_{i,ext}^N	&:=&  \int_{\partial T_{\kappa}^N} \big\{ \partial_n  U^{a^N}[w(h_i^N)](x-h_i^N) - P^{a^N}[w(h_i^N)](x-h_i^N)n\big\} \cdot u^N(x) {\rm d}\sigma \,.
\end{eqnarray*} 
Recalling that $(U^{a^N},P^{a^N})$ is the solution to the Stokes problem in the exterior of a ball of radius $a^N$, and that $v_i^N$
is constant on $\partial B_i^N,$ we have an explicit value for the interior integral whatever the value of the index $i$ (see \eqref{eq_stokesformula}):
$$
I_{i,int}^N = -{6\pi}a^N w(h_i^N) \cdot v_i^N\,.
$$
For the other term, we apply that the diameter of $T^{N}_{\kappa}$ is small so that we may approximate $u^N$ on $\partial T^{N}_{\kappa}$ by a constant. Namely, we choose:
$$
\bar{u}^N_{\kappa} = \dfrac{1}{|[T^{N}_{\kappa}]_{2\delta}|}\int_{[T^{N}_{\kappa}]_{2\delta}} u^N(x){\rm d}x,
$$
where $[T^{N}_{\kappa}]_{2\delta}$ is the $\lambda^N/(2\delta)$-neighborhood of $\partial T_{\kappa}^N$
inside $\mathring{T}_{\kappa}^N.$ At this point, we remark that we have actually two notations for the same quantity. Indeed, a simple draw shows that introducing
$x_{\kappa}^N$ the center of $T_{\kappa}^N,$ we have:
$$
\mathring{T}^{N}_{\kappa} = B_{\infty}\left(x_{\kappa}^N,\frac{\lambda^N}{2}\right) \quad \text{ while } \quad [T^{N}_{\kappa}]_{2\delta} = A\left(x_{\kappa}^N,\left[1-\frac 1{\delta}\right]\frac{\lambda^N}{2},\frac{\lambda^N}{2}\right)\,.
$$
So, we replace:
\begin{eqnarray*}
I^N_{i,ext} &=& \int_{\partial T_{\kappa}^N} \big\{ \partial_n  U^{a^N}[w(h_i^N)](x-h_i^N) - P^{a^N}[w(h_i^N)](x-h_i^N)n\big\} \cdot \bar{u}^N_{\kappa} {\rm d}\sigma   \Bigg)\\
			&& \; + \int_{\partial T_{\kappa}^N} \big\{ \partial_n  U^{a^N}[w(h_i^N)](x-h_i^N) - P^{a^N}[w(h_i^N)](x-h_i^N)n\big\} \cdot (u^N(x) - \bar{u}^N_{\kappa}) {\rm d}\sigma\,.
\end{eqnarray*}

For the first term on the right-hand side of this last identity, we apply that the flux through hypersurfaces of the normal stress tensor is conserved by solutions to the Stokes problem so that, applying \eqref{eq_stokesformula}, we have:
\begin{multline*}
\int_{\partial T_{\kappa}^N} \big\{ \partial_n  U^{a^N}[w(h_i^N)](x-h_i^N) - P^{a^N}[w(h_i^N)](x-h_i^N)n\big\}{\rm d}\sigma   \Bigg) \\
\begin{array}{rcl}
&=&  \displaystyle \int_{\partial B_i^N} \big\{ \partial_n  U^{a^N}[w(h_i^N)](x-h_i^N) - P^{a^N}[w(h_i^N)](x-h_i^N)n\big\}  {\rm d}\sigma \\[12pt]
&=& - \displaystyle {6\pi} a^N w(h_i^N).
\end{array}
\end{multline*}
Finally, we obtain:
\begin{equation} \label{eq_INkappa}
\tilde{I}^N_{\kappa} = 6\pi Na^N \left[ \dfrac{1}{N} \sum_{i \in \mathcal I^N_{\kappa} \setminus \mathcal Z^N_{\delta}} ( w(h_i^N) \cdot v_i^N - w(h_i^N) \cdot \bar{u}_{\kappa}^N )\right]+ Err_{\kappa}
\end{equation}
with: 
$$
Err_{\kappa} = \int_{\partial T_{\kappa}^N}  \left\{ \sum_{i \in \mathcal I^{N}_{\kappa} \setminus \mathcal Z_{\delta}^N} 
 \partial_n  U^{a^N}[w(h_i^N)](\cdot-h_i^N) - P^{a^N}[w(h_i^N)](\cdot-h_i^N)n\right\} \cdot (u^N - \bar{u}^N_{\kappa}) {\rm d}\sigma\,.
$$
We control this error term with the following lemma:
\begin{lemma} \label{lem_errkappa}
There exists a constant $C_{\delta}$ depending only on $\delta$ such that,
$$
|Err_{\kappa}| \lesssim {C_{\delta}}|\lambda^N|^{\frac 52} \|\nabla u^N\|_{L^2(T_{\kappa}^N)} \,, \quad \forall \, \kappa \in \Kappa^N\,.
$$
\end{lemma}
\begin{proof}
For large $N$, we have that 
$$
[T^{N}_{\kappa}]_{2\delta} \subset T_{\kappa}^N \setminus  \bigcup_{i \in \mathcal I_{\kappa}^N \setminus \mathcal Z_{\delta}^N} \overline{B_{i}^N}.  
$$ 
Indeed, $B_{i}^N = B(h_{i}^N,a^N)$ and, for $i \in \mathcal I_{\kappa}^N \setminus \mathcal Z_{\delta}^N$ we have
that $h_{i}^N$ is $\lambda^N/\delta $ far from $\partial T_{\kappa}^N.$ These centers are thus $\lambda^N/(2\delta)$ 
far from $[T_{N}^{\kappa}]_{2\delta}$ which is larger than $a^N$ for large $N$. In particular all the stokeslets
in $w^s_{\kappa}$ satisfy: 
\begin{equation} \label{eq_remarkstokes}
\left\{ \begin{array}{rcl}
-\Delta U^{a^N}[w(h_i^N)](x-h_i^N) + \nabla P^{a^N}[w(h_i^N)](x-h_i^N) &=& 0\,, \\[8pt]
{\rm div}\, U^{a^N}[w(h_i^N)](x-h_i^N) &=& 0\,,
\end{array}
\right.
\quad \text{ on $[T_{N}^{\kappa}]_{2\delta}\,.$} 
\end{equation}
Consequently, we split 
$$
\partial [T^{N}_{\kappa}]_{2\delta} = \partial T_{\kappa}^N \cup \partial T^N_{\kappa,\delta}
$$
where 
$$
\partial T^{N}_{\kappa,\delta} = \{x \in T_{\kappa}^N  \text{ s.t. } \text{dist} (x, \partial T_{\kappa}^N) = \lambda^N/(2\delta)\}\,.
$$
We remark then that for any divergence-free $v \in H^1([T_{\kappa}^N]_{2\delta})$ satisfying 
$$
\left\{
\begin{array}{rcll}
v &=& u^N - \bar{u}^N_{\kappa} \,, & \text{ on $\partial T_{\kappa}^N$}\,,\\[8pt]
v &=& 0  \,, & \text{ on $\partial T^N_{\kappa,\delta}$}\,, 
\end{array}
\right.
$$
integrating by parts $Err_{\kappa}$ and applying \eqref{eq_remarkstokes},  we have:
$$
Err_{\kappa} = \int_{[T_{\kappa}^N]_{2\delta}}  \left\{ \sum_{i \in \mathcal I^{N}_{\kappa} \setminus \mathcal Z_{\delta}^N} 
 \nabla  U^N[w(h_i^N)](\cdot-h_i^N) \right\} : \nabla v \,,
$$
so that:
\begin{equation} \label{eq_preskekappa}
|Err_{\kappa}| \leq  \left\{\sum_{i \in \mathcal I^{N}_{\kappa} \setminus \mathcal Z_{\delta}^N} 
 \|\nabla  U^N[w(h_i^N)](\cdot-h_i^N)\|_{L^2([T_{\kappa}^N]_{2\delta})} \right\}  \|\nabla v\|_{L^2([T_{\kappa}^N]_{2\delta})}\,.
\end{equation}
Let choose a suitable $v$ in order to apply this estimate. We recall that we introduced $x^N_{\kappa}$ the center
of $T^N_{\kappa}$ and that we remarked  that
$$
T^{N}_{\kappa} = B_{\infty}\left(x^N_{\kappa},\dfrac {\lambda^N}{2} \right) \,, \quad [T^N_{\kappa}]_{2\delta} = A \left( x^N_{\kappa},\left[1-\dfrac{1}{\delta} \right]\frac{\lambda^N}{2},\frac {\lambda^N}{2}\right)\,. 
$$
We set
\begin{multline*}
v(x) = \zeta_{\delta}((x-x^N_{\kappa})/\lambda^N) (u^N(x)-\bar{u}_{\kappa}^N) \\
 - \mathfrak B_{x^N_{\kappa},(1-1/\delta)\lambda^N/2,\lambda^N/2}[ x \mapsto (u^N(x)-\bar{u}_{\kappa}^N) \cdot \nabla [\zeta_{\delta}((x-x^N_{\kappa})/\lambda^N)]]\,.
\end{multline*}
Again $v$ is well-defined as one shows by direct computations that the argument of the Bogovskii operator has mean zero on 
$ A(x^N_{\kappa},(1-1/\delta)\lambda^N/2,\lambda^N/2)).$ Applying Lemma \ref{lem_div}, we have then that there exists a constant
$C_{\delta}$ depending only on $\delta$ for which:
$$
\|\nabla v\|_{L^2([T^N_{\kappa}]_{2\delta})} \leq C_{\delta} \left[  \|\nabla u^N\|_{L^2([T^N_{\kappa}]_{2\delta})} + \frac 1{\lambda^N}\|u^N(x)-\bar{u}_{\kappa}^N\|_{L^2([T^N_{\kappa}]_{2\delta})} \right].
$$
Here we note that the $\bar{u}_{\kappa}^N$ is exactly the average of $u^N$ on $[T^N_{\kappa}]_{2\delta}.$ Consequently, applying the Poincar\'e-Wirtinger inequality in the annulus $[T^N_{\kappa}]_{2\delta}$ with the remark on the best constant as in Lemma \ref{lem_PW} we obtain finally that:
\begin{equation} \label{eq_controlliftukappa}
\|\nabla v\|_{L^2([T^N_{\kappa}]_{2\delta})} \leq C_{\delta}  \|\nabla u^N\|_{L^2([T^N_{\kappa}]_{2\delta})}\,.
\end{equation}
As for the stokeslet, we remark again that for any $i \in \mathcal I_{\kappa}^N \setminus \mathcal Z_{\delta}^N$ the minimum distance between  $h_{i}^N$ and $[T_{\kappa}^N]_{2\delta}$ is larger than $\lambda^N/(2\delta).$ Hence, applying the expansion \eqref{eq_Stokesexpansion} of the stokeslet $U^{a^N}[w(h_i^N)]$ we obtain that
\begin{eqnarray*}
\|\nabla U^{a^N}[w(h_i^N)](\cdot -h_{i}^N)\|_{L^2([T_{\kappa}^N]_{2\delta})} &\leq&  \left(\int_{\lambda^N/(2\delta)}^{\infty} \dfrac{|a^N|^2{\rm d}r}{r^2} \right)^{\frac 12}|w(h_i^N)| \\
&\lesssim & \dfrac{\sqrt{2\delta}}{N\sqrt{\lambda^N}}\,.
\end{eqnarray*} 
Combining these computations for the (at most) $M_{\kappa}^N$ indices $i \in \mathcal I^{N}_{\kappa} \setminus \mathcal Z_{\delta}^N$   entails by \eqref{eq_propMkappa} that:
\begin{equation} \label{eq_controlstokeslet}
\sum_{i \in \mathcal I^{N}_{\kappa} \setminus \mathcal Z_{\delta}^N} 
 \|\nabla  U^N[w(h_i^N)](\cdot-h_i^N)\|_{L^2([T_{\kappa}^N]_{2\delta})} \lesssim \sqrt{2\delta}|\lambda^N|^{\frac 52}\,.
\end{equation}
Combining \eqref{eq_controlliftukappa} and \eqref{eq_controlstokeslet} in \eqref{eq_preskekappa} yields the expected result.  
\end{proof}

\medskip

Summing \eqref{eq_INkappa} over $\kappa,$ we obtain that:
\begin{eqnarray}
\tilde{I}^N &=&  6\pi Na^N \left[ \dfrac{1}{N} \sum_{\kappa \in \Kappa^N} \sum_{i \in \mathcal I^N_{\kappa} \setminus \mathcal Z^N_{\delta}} ( w(h_i^N) \cdot v_i^N - w(h_i^N) \cdot \bar{u}_{\kappa}^N) \right] + Err \notag \\
			&=& 6\pi Na^N \left[	\dfrac{1}{N} \sum_{i \in \mathcal I^N \setminus \mathcal Z^N_{\delta}} w(h_i^N) \cdot v_i^N 
- \dfrac{1}{N} \sum_{\kappa \in \Kappa^N} \sum_{i \in \mathcal I^N_{\kappa} \setminus \mathcal Z^N_{\delta}} w(h_i^N) \cdot \bar{u}_{\kappa}^N \right] + Err \label{eq_INcomplet}
\end{eqnarray}
where 
$$
Err  = \sum_{\kappa \in \Kappa^N} Err_{\kappa}.
$$
Hence, applying Lemma \ref{lem_errkappa},  a Cauchy-Schwarz inequality and remarking again that the $(T^{N}_{\kappa})_{\kappa \in \Kappa^N}$  form a partition of a subset of $\Omega$ with a number of elements satisfying \eqref{eq_propKappaN}, we have: 
\begin{eqnarray}
\notag
|Err| &\lesssim &C_{\delta} \sum_{\kappa \in \Kappa^N} {\|\nabla u^N\|_{L^2(T^N_{\kappa})}} |\lambda^N|^{\frac 52}  
		 \lesssim  {C_{\delta}} \lambda^N \|\nabla u^N\|_{L^2(\Omega)} \\
		&\lesssim&   {C_{\delta}} \lambda^N  \,.\label{eq_err}
\end{eqnarray}
As $\lambda^N \to 0,$ the asymptotics of $\tilde{I}^N$ is given by the two first terms on the right-hand side of \eqref{eq_INcomplet}.
We know by assumption that $Na^N \to a.$ So, we make precise now the asymptotics of the two remaining terms in the two following lemmas:
\begin{lemma}
Given $\delta \geq 4,$ there holds:
$$
\limsup_{N\to \infty} \left|\dfrac{1}{N} \sum_{i \in \mathcal I^N \setminus \mathcal Z^N_{\delta}} w(h_i^N) \cdot v_i^N  - \int_{\Omega} j(x) \cdot w(x) {\rm d}x \right| \lesssim \dfrac{1}{\delta}\,. 
$$
\end{lemma}
\begin{proof}
As $w \in C^{\infty}_c(\Omega)$ and $(T_{\kappa}^N)_{\kappa \in \Kappa^N}$ is a covering of $\text{Supp}(w)$ we have by assumption \eqref{eq_ass5} that:
$$
\int_{\Omega} j(x) \cdot w(x){\rm d}x =\lim_{N \to \infty} \dfrac 1N \sum_{i=1}^N w(h_i^N) \cdot v_i^N =  \lim_{N \to \infty}  \dfrac 1N \sum_{i \in \mathcal I^N} w(h_i^N) \cdot v_i^N \,.
$$
Hence, our proof reduces to find a uniform bound on 
$$
\dfrac 1N \left[ \sum_{i \in \mathcal I^N} w(h_i^N) \cdot v_i^N - \sum_{i \in \mathcal I^N \setminus \mathcal Z^N_{\delta} } w(h_i^N) \cdot v_i^N\right] = 
\dfrac 1N \left[ \sum_{i \in \mathcal Z^N_{\delta} \cap \, \mathcal I^N} w(h_i^N) \cdot v_i^N \right]\,.
$$
However, for large $N,$ there holds:
$$
\left|\dfrac{1}{N} \sum_{i \in \mathcal Z_{\delta}^N \cap \, \mathcal I^N} w(h_i^N) \cdot v_i^N \right| \leq    \left( \dfrac{1}{N} \sum_{i\in\mathcal Z_{\delta}^N} |v_i^N|^2 \right)^{\frac 12} \left(\dfrac{1}{N} \sum_{i \in \mathcal Z_{\delta}^N } |w(h_i^N)|^2 \right)^{\frac 12}\,.
$$
Here, we apply \eqref{eq_couloir} that has guided our choice for the covering $(T_{\kappa}^N)_{\kappa \in \Kappa^{N}}:$
\begin{eqnarray*}
\left( \dfrac{1}{N} \sum_{i\in\mathcal Z_{\delta}^N} |v_i^N|^2 \right) & \leq  & \dfrac{12}{\delta}  \left( 1 + | \mathcal E^{\infty} |^2\right)\,,  \\ 
\left( \dfrac{1}{N} \sum_{i\in\mathcal Z_{\delta}^N} |w(h_i^N)|^2 \right) & \leq &  \dfrac{12}{\delta}   \left( 1 + | \mathcal E^{\infty} |^2\right) \|w\|^2_{L^{\infty}} \,. \\ 
\end{eqnarray*}
Combining these two estimates, we obtain:
$$
\limsup_{N \to \infty} \left|\dfrac{1}{N} \sum_{i \in \mathcal Z_{\delta}^N \cap \, \mathcal I^N } w(h_i^N) \cdot v_i^N \right| \leq 
\dfrac{12}{\delta}  \left( 1 + |\mathcal E^{\infty}|^2\right)\|w\|_{L^{\infty}}\,.
$$
\end{proof}

\begin{lemma}
For $\delta \geq 4$ there holds:
$$
\limsup_{N\to\infty} \left|\dfrac{1}{N} \sum_{\kappa \in \Kappa^N} \sum_{i \in \mathcal I^N_{\kappa} \setminus \mathcal Z^N_{\delta}} w(h_i^N) \cdot \bar{u}^N_{\kappa}   -  \int_{\Omega} \rho(x) \bar{u}(x) \cdot w(x) {\rm d}x\right| \lesssim\dfrac{1}{\sqrt{\delta}} \,.
$$
\end{lemma}
\begin{proof} As in the previous proof, let first complete the sum by reintroducing the $\mathcal Z_{\delta}^N$ indices:
\begin{equation} \label{eq_rewriterho}
\dfrac{1}{N} \sum_{\kappa \in \Kappa^N} \sum_{i \in \mathcal I^N_{\kappa} \setminus \mathcal Z^N_{\delta}} w(h_i^N) \cdot \bar{u}^N_{\kappa} 
= \dfrac{1}{N} \sum_{\kappa \in \Kappa^N} \sum_{i \in \mathcal{I}^N_{\kappa}} w(h_i^N) \cdot \bar{u}^N_{\kappa}  + \tilde{E}rr^{N}
\end{equation}
where:
$$
\tilde{E}rr^N = \dfrac{1}{N} \sum_{\kappa \in \Kappa^N} \sum_{i \in \mathcal I^N_{\kappa} \cap \mathcal Z^N_{\delta}} w(h_i^N) \cdot \bar{u}^N_{\kappa} .
$$
For the first term on the right-hand side of \eqref{eq_rewriterho},  we remark that:
$$
\dfrac{1}{N} \sum_{\kappa \in \Kappa^N} \sum_{i \in \mathcal I^N_{\kappa}} w(h_i^N) \cdot \bar{u}^N_{\kappa} 
=
\left(1 - \left(1-\frac 1\delta\right)^{3}\right)^{-1} \dfrac{1}{|\lambda^N|^3 N}\sum_{\kappa \in \Kappa^N} \int_{[T_{\kappa}^N]_{2\delta}} \left(\sum_{i \in \mathcal I^N_{\kappa}} w(h_i^N) \right) \cdot u^N\,.
$$
So, we introduce:
$$
\sigma^N= 
\left(1 - \left(1-\frac 1\delta\right)^{3}\right)^{-1} \dfrac{1}{N|\lambda^N|^3} \sum_{\kappa \in \Kappa^N}  \left(\sum_{i \in \mathcal I^N_{\kappa}} w(h_i^N) \right) \mathbf{1}_{[T_{\kappa}^N]_{2\delta}}\,,
$$
for which:
$$
\dfrac{1}{N} \sum_{\kappa \in \Kappa^N} \sum_{i \in \mathcal I^N_{\kappa}} w(h_i^N) \cdot \bar{u}^N_{\kappa} 
= \int_{\Omega} \sigma^N(x) \cdot u^N(x){\rm d}x. 
$$

\medskip

On the one-hand, we note that:
$$
\|\sigma^N\|_{L^1(\Omega)} \leq \dfrac{1}{N} \sum_{\kappa \in \Kappa^N} M^N_{\kappa} \|w\|_{L^{\infty}} \,,
$$
where $\sum_{\kappa \in \Kappa^N} M^N_{\kappa} \leq N,$ so that:  
$$
\|\sigma^N\|_{L^1(\Omega)} \leq  \|w\|_{L^{\infty}} \,.
$$
Complementarily, because of assumption \eqref{eq_ass_concentration}, we also have :
\begin{eqnarray*}
\|\sigma^N\|_{L^{\infty}(\Omega)} & \leq&  \left(1 - \left(1-\frac 1\delta\right)^{3}\right)^{-1} \sup_{\kappa \in \Kappa^N} \dfrac{M^N_{\kappa}}{N|\lambda^N|^3} \|w\|_{L^{\infty}} \\
									&\lesssim &  \left(1 - \left(1-\frac 1\delta\right)^{3}\right)^{-1} \,,
\end{eqnarray*} 
and $\sigma^N$ is bounded in all $L^q$-spaces. 

\medskip

On the other hand, for any $v \in C^{\infty}_c(\Omega)$ we have 
$$
\int_{\Omega} \sigma^N (x)\cdot v(x){\rm d}x =  \dfrac{1}{N} \sum_{\kappa \in \Kappa^N} \sum_{i \in \mathcal I_{\kappa}^N} w(h_i^N) \cdot \bar{v}_{\kappa}^N
$$
with 
$$
 \bar{v}_{\kappa}^N = \dfrac{1}{|[T^{N}_{\kappa}]_{2\delta}|} \int_{[T_{N}^{\kappa}]_{2\delta}} v(x){\rm d}x.
$$
We remark that, for any $i \in \mathcal I_{\kappa}^N,$ $h_i^N$ is inside $T^{N}_{\kappa}$ whose diameter is $\lambda^N.$
This entails:
$$
\left| \bar{v}_{\kappa}^N  - v(h_i^N) \right| \lesssim \lambda^N \|\nabla v\|_{L^{\infty}} .
$$
Gathering these identities for all indices $i$ in all the cubes $T_{\kappa}^N,$ we infer  : 
$$
\left| \int_{\Omega} \sigma^N(x) \cdot v(x){\rm d}x -   \dfrac{1}{N}\sum_{i \in \mathcal I^N} w(h_i^N) \cdot v(h_i^N) \right| \lesssim  \lambda^N \|\nabla v\|_{L^{\infty}}\|w\|_{L^{\infty}}\,.
$$
Consequently, assumption \eqref{eq_ass4} implies that:
$$
\lim_{N \to \infty} \int_{\Omega} \sigma^N(x) \cdot v(x){\rm d}x = \int_{\Omega} \rho(x) w(x) \cdot v(x){\rm d}x\,,
$$
and $\sigma^N \rightharpoonup \rho w$ weakly in $L^q(\Omega)$ for arbitrary $q \in (1,\infty)$. 
Combining then the weak convergence of $\sigma^N$ in $L^{2}(\Omega)$ and the strong convergence of $u^N$ in $L^2(\Omega)$
(up to the extraction of a subsequence), we have:
$$
\lim_{N \to \infty} \int_{\Omega} \sigma^N \cdot u^N = \int_{\Omega} \rho w \cdot \bar{u}.
$$

\medskip

As for the remainder term, we introduce:
$$
\tilde{\sigma}^N  = \left(1 - \left(1-\frac 1\delta\right)^{3}\right)^{-1} \dfrac{1}{N|\lambda^N|^3} \sum_{\kappa \in \Kappa^N}  \left(\sum_{i \in \mathcal I^N_{\kappa} \cap \mathcal Z^{N}_{\delta}} |w(h_i^N)| \right) \mathbf{1}_{[T_{\kappa}^N]_{2\delta}}\,.
$$
so that:
$$
|\tilde{E}rr^N| \leq \int_{\Omega} \tilde{\sigma}^N(x) |u^N(x)|{\rm d}x\,.
$$
With similar arguments as in the previous computations,  we have, applying \eqref{eq_couloir}:
$$
\|\tilde{\sigma}^N\|_{L^1(\Omega)} \leq \dfrac{1}{N} \card \mathcal Z_{\delta}^N \|w\|_{L^\infty} \lesssim \dfrac{1}{\delta}.
$$
Furthermore, we have:
$$
\|\tilde{\sigma}^N\|_{L^{\infty}(\Omega)} \lesssim  \delta .
$$
Consequently, by interpolation, we obtain:
$$
\|\tilde{\sigma}^N\|_{L^{\frac 43}(\Omega)} \lesssim  \dfrac{ 1}{\sqrt{\delta}} \,.
$$
As $u^N$ is bounded in $L^{4}(\Omega)$ by Sobolev embedding, this yields that:
$$
|\tilde{E}rr^N| \lesssim \dfrac{1}{\sqrt{\delta}}.
$$
This ends the proof.
\end{proof}

\section{Two (counter-)examples} 

In this paper, we derive the Stokes-Brinkman system by homogenizing the Stokes problem in a perforated domain. 
Our main result is valid only in the dilution regime specified by assumptions \eqref{eq_ass_dmin}-\eqref{eq_ass_concentration}. 

Assumption \eqref{eq_ass_dmin} is critical to our computation. It implies that, zooming around one of the holes, the solution to the $N$-hole problem looks alike the solution in the exterior of one hole. Then, the action of the holes on the flow can be computed by adding the contribution of all the holes as if they were alone in the fluid (but with a non-trivial speed at infinity). 
In the first part of this section, we discuss what happens when this assumption degenerates and $d_{min}^N$ decays like $1/N.$

Assumption \eqref{eq_ass_concentration} is motivated by the fact that we want to consider particle distribution functions 
$(x,v) \mapsto f(x,v)$  such that  the associated density $x \mapsto \rho(x)$ is bounded. This implies that, 
for arbitrary $\lambda > 0$ the density of  particles in balls of radius $\lambda$ satisfies
$$
\sup_{x \in \Omega}\, \langle \rho, \mathbf{1}_{B(x,\lambda)} \rangle \leq \|\rho\|_{L^{\infty}} \lambda^3.
$$
One may prove that under the sole assumption \eqref{eq_ass4}, {\em i.e.} the sequence of discrete density measures $\rho^N$ converges to $\rho(x) {\rm d}x$ with $\rho \in L^{\infty}(\Omega),$  implies  that there exists a sequence $(\lambda^N)_{N\in\mathbb N}$ converging to $0$  for which:
$$
\sup_{N\in \mathbb N}\dfrac{1}{|\lambda^N|^3}\sup_{x \in \Omega}\, \langle \rho^N, \mathbf{1}_{B(x,\lambda^N)} \rangle < \infty\,.
$$
Assumption \eqref{eq_ass_concentration} require this property for a particular sequence $(\lambda^N)_{N\in \mathbb N}.$ As mentioned in the introduction, the goal is to fix this assumption for the largest sequence $(\lambda^N)_{N\in \mathbb N}$ possible. So in the second part of this section, we discuss the optimality of the sequence given by \eqref{eq_ass_lambdaN}.

\medskip

\subsection{On assumption \eqref{eq_ass_dmin}}
If $d_{min}^N \sim 1/N$ or $d_{min}^N << 1/N$, 
 the distance between holes becomes comparable to their common radius and the influence of the holes on the solution is more intricate. 
In such a case, we expect that one can pack the holes into sub-groups containing holes between which the distance is smaller or comparable to their common radius. Then, each of these packs has to be considered  as one hole with a complicated shape instead of a group of holes. 

This remark applies in the following example.  Let divide the container $\Omega = [0,1]^3$ into $N/2$ cubes $(T_k^N)_{k=1,\ldots,N/2}$ of width $(2/N)^{1/3}.$ Each of the cubes contains $2$ holes so that the centers of these holes are diametrically symmetric on a sphere of radius  $(1+h)/N$ ($h$ is a positive parameter) centered in the center of the cube (see Figure \ref{FPC}). The geometry is then completly fixed by the set of vectors $(h_{k})_{k=1,\ldots,N/2}$ linking the centers of two spheres in the same cube.  Broadly, it comes from the proof in the previous sections that the Brinkman term in the limit problem can be computed by zooming in any of the elementary cells (with a scale $1/N$), computing the drag terms involved by the Stokes problem in the cells and summing them after rescaling. In this example, one cell corresponds to a cube $T_k^N$ which contains two spherical holes. Then, the drag term is computed by considering the Stokes problem in an exterior domain whose shape is the complement of two unit balls.  We expect that, summing these contributions, the resulting Brinkman term has a different structure than  "$6\pi(j - \rho u)$". Especially, it should depend nonlinearly on the  $(h_k)_{k=1,\ldots,N/2}$ and anisotropically on $u$.
Such computations are handled in \cite{MSH} to which we refer for more details.

 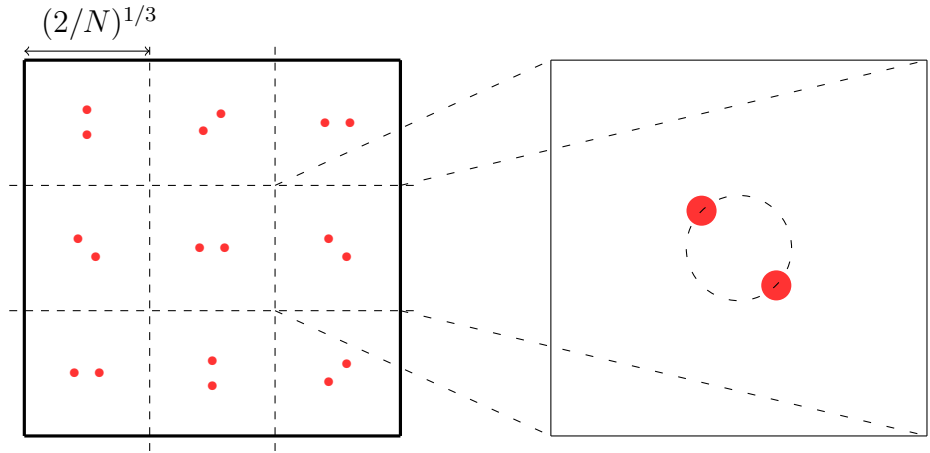
\begin{figure}[h]
 \begin{center}
\begin{tikzpicture}

%% Le carre
\draw[very thick] (0,0) -- (5,0); 
\draw[very thick] (0,0) -- (0,5); 
\draw[very thick] (0,5) -- (5,5); 
\draw[very thick] (5,0) -- (5,5); 

%% Les points

\draw[color=red!80] (5/6+1/6*1,5/6+1/6*0) node{\tiny$\bullet$} ;
\draw[color=red!80] (5/6-1/6*1,5/6+1/6*0) node{\tiny$\bullet$} ;
\draw[color=red!80] (5/3+5/6+1/6*0,5/6+1/6*1) node{\tiny$\bullet$} ;
\draw[color=red!80] (5/3+5/6+1/6*0,5/6-1/6*1) node{\tiny$\bullet$} ;
\draw[color=red!80] (10/3+5/6+1/6*.71,5/6+1/6*.71) node{\tiny$\bullet$} ;
\draw[color=red!80] (10/3+5/6-1/6*.71,5/6-1/6*.71) node{\tiny$\bullet$} ;

\draw[color=red!80] (5/6-1/6*.71,5/3+5/6+1/6*.71) node{\tiny$\bullet$} ;
\draw[color=red!80] (5/6+1/6*.71,5/3+5/6-1/6*.71) node{\tiny$\bullet$} ;
\draw[color=red!80] (5/3+5/6+1/6*1,5/3+5/6+1/6*0) node{\tiny$\bullet$} ;
\draw[color=red!80] (5/3+5/6-1/6*1,5/3+5/6-1/6*0) node{\tiny$\bullet$} ;
\draw[color=red!80] (10/3+5/6+1/6*.71,5/3+5/6-1/6*.71) node{\tiny$\bullet$} ;
\draw[color=red!80] (10/3+5/6-1/6*.71,5/3+5/6+1/6*.71) node{\tiny$\bullet$} ;

\draw[color=red!80] (5/6+1/6*0,10/3+5/6+1/6*1) node{\tiny$\bullet$} ;
\draw[color=red!80] (5/6-1/6*0,10/3+5/6-1/6*1) node{\tiny$\bullet$} ;
\draw[color=red!80] (5/3+5/6+1/6*.71,10/3+5/6+1/6*.71) node{\tiny$\bullet$} ;
\draw[color=red!80] (5/3+5/6-1/6*.71,10/3+5/6-1/6*.71) node{\tiny$\bullet$} ;
\draw[color=red!80] (10/3+5/6+1/6*1,10/3+5/6+1/6*0) node{\tiny$\bullet$} ;
\draw[color=red!80] (10/3+5/6-1/6*1,10/3+5/6-1/6*0) node{\tiny$\bullet$} ;

%% La subdivision
\draw[dashed] (5/3,-0.2) -- (5/3,5.2) ;
\draw[dashed] (10/3,-0.2) -- (10/3,5.2) ;
\draw[dashed] (-0.2,10/3) -- (5.2,10/3) ;
\draw[dashed] (-0.2,5/3) -- (5.2,5/3) ;

\draw[<->] (0,5.12) -- (5/3,5.12) ; 
 
\draw (1,5.12) node [above] {$(2/N)^{1/3}$} ; 

\draw [loosely dashed] (5,5/3) -- (12,0) ; 
\draw [loosely dashed] (5,10/3) -- (12,5) ; 
\draw [loosely dashed] (10/3,5/3) -- (7,0) ; 
\draw [loosely dashed] (10/3,10/3) -- (7,5) ;

\draw (7,0) -- (12,0); 
\draw (7,0) -- (7,5); 
\draw (7,5) -- (12,5); 
\draw (12,0) -- (12,5); 

\fill[color = red!80] (19/2-.7*.71,5/2+.7*.71) circle (.2) ; 
\fill[color = red!80] (19/2+.7*.71,5/2-.7*.71) circle (.2) ; 
\draw[loosely dashed] (19/2,5/2) circle (.7) ; 

{
}
\end{tikzpicture}
\caption{First counter-example configuration \label{FPC}}  
\end{center}

\end{figure}

\subsection{On assumption \eqref{eq_ass_concentration}}
Our next example is a variant of the construction in \cite{Allaire}. In particular, 
 we go back to the case of a Stokes problem in a bounded perforated domain 
with a source term $f \in L^2(\Omega).$ We consider vanishing boundary conditions on the holes for simplicity. The holes will be distributed (almost) periodically so that their density converges to a uniform distribution in $[0,1]^3.$ In particular, if our main result were extending to this case, the homogenized system should read:
$$
\begin{array}{rcl}
-\Delta u + \nabla p &=& f- 6\pi \mathbf{1}_{[0,1]^3} u \\
{\rm div} u&=&  0 
\end{array}
\quad \text{ in $\Omega.$}
$$

Nevertheless, let consider $\Omega$ a smooth bounded domain containing $[0,1]^3$ and $(P^N)_{N\in\mathbb N}$ a diverging sequence of integers. We  assume that 
 \begin{equation} \label{eq_geom1}
\lim_{N \to \infty} \dfrac{P^N }{N}  = 0.
 \end{equation}
 Given $N \in \mathbb N$ we cover $\mathbb R^3$ with disjoint cubes $(T_k^N)_{k \in \mathbb Z^3}$ of width $\sigma^N = |P^N/N|^{1/3}.$ For $k \in \mathbb Z^3,$
 we denote by $x_k^N$ the center of $T_k^N$ so that $T_k^N = B_{\infty}(x_k^N,\sigma^N/2).$ For $N$
 sufficiently large, we extract a list 
 $\mathcal K^N$ containing $\lfloor N/P^N \rfloor +1$  indices of  cubes $T_k^N$ that are inside $\Omega.$ To do this, we first choose all the cubes that are included in $[0,1]^3$ and we complement the list by choosing at most
 one other cube that is included in $\Omega.$  
 For the $\lfloor N/P^N \rfloor$ first cubes of the list $\Kappa^N$ (including all the ones that are inside $[0,1]^3$), we perform $P^N$ holes in $T_k^N.$
 The holes are distributed concentrically  around the center $x_k^N$ of $T_k^N$ on an orthogonal grid of step $2d_m^N > 0.$ In particular, we center the grid
 so that all the  perforated sites are inside $B_{\infty}(x_k^N,  \lfloor (|P^N|^{1/3} +1) \rfloor d_m^N).$ We  assume below that:
 \begin{equation}\label{eq_geom2} 
 \lim_{N \to \infty} N^{\frac 13} d_m^N = 0\,, \qquad \lim_{N\to \infty} N d_{m}^N = +\infty\,. 
 \end{equation}
 The first part of this asumption entails that, for $N$ large, all the holes around $x_{k}^N$ are inside $T_{k}^N.$
 In the last cube, we perform $N - \lfloor N/P^N \rfloor P^N$ holes in the same way so that we have eventually $N$ holes of radius $1/N$
 in $\Omega$  that we label $(B(h_i^N,1/N))_{i=1,\ldots,N}$. See Figure \ref{SPC} for an illustration.

\medskip

\begin{figure}[h]

\begin{center}
\begin{tikzpicture}

%% Le carre
\draw[very thick] (0,0) -- (10,0); 
\draw[very thick] (0,0) -- (0,10); 
\draw[very thick] (0,10) -- (10,10); 
\draw[very thick] (10,0) -- (10,10); 

%% Les points
\draw[color=red!80] (4/3,4/3) node{\tiny$\bullet$} ;
\draw[color=red!80] (4/3,5/3) node{\tiny$\bullet$} ;
\draw[color=red!80] (4/3,2) node{\tiny$\bullet$} ;
\draw[color=red!80] (5/3,+4/3) node{\tiny$\bullet$} ;
\draw[color=red!80] (5/3,+5/3) node{\tiny$\bullet$} ;
\draw[color=red!80] (5/3,+2) node{\tiny$\bullet$} ;
\draw[color=red!80] (2,+4/3) node{\tiny$\bullet$} ;
\draw[color=red!80] (2,+5/3) node{\tiny$\bullet$} ;
\draw[color=red!80] (2,+2) node{\tiny$\bullet$} ;
\draw[color=red!80] (10/3+4/3,+4/3) node{\tiny$\bullet$} ;
\draw[color=red!80] (10/3+4/3,+5/3) node{\tiny$\bullet$} ;
\draw[color=red!80] (10/3+4/3,+2) node{\tiny$\bullet$} ;
\draw[color=red!80] (10/3+5/3,+4/3) node{\tiny$\bullet$} ;
\draw[color=red!80] (10/3+5/3,+5/3) node{\tiny$\bullet$} ;
\draw[color=red!80] (10/3+5/3,+2) node{\tiny$\bullet$} ;
\draw[color=red!80] (10/3+2,+4/3) node{\tiny$\bullet$} ;
\draw[color=red!80] (10/3+2,+5/3) node{\tiny$\bullet$} ;
\draw[color=red!80] (10/3+2,+2) node{\tiny$\bullet$} ;
\draw[color=red!80] (20/3+4/3,+4/3) node{\tiny$\bullet$} ;
\draw[color=red!80] (20/3+4/3,+5/3) node{\tiny$\bullet$} ;
\draw[color=red!80] (20/3+4/3,+2) node{\tiny$\bullet$} ;
\draw[color=red!80] (20/3+5/3,+4/3) node{\tiny$\bullet$} ;
\draw[color=red!80] (20/3+5/3,+5/3) node{\tiny$\bullet$} ;
\draw[color=red!80] (20/3+5/3,+2) node{\tiny$\bullet$} ;
\draw[color=red!80] (20/3+2,+4/3) node{\tiny$\bullet$} ;
\draw[color=red!80] (20/3+2,+5/3) node{\tiny$\bullet$} ;
\draw[color=red!80] (20/3+2,+2) node{\tiny$\bullet$} ;

\draw[color=red!80] (4/3,10/3+4/3) node{\tiny$\bullet$} ;
\draw[color=red!80] (4/3,10/3+5/3) node{\tiny$\bullet$} ;
\draw[color=red!80] (4/3,10/3+2) node{\tiny$\bullet$} ;
\draw[color=red!80] (5/3,10/3+4/3) node{\tiny$\bullet$} ;
\draw[color=red!80] (5/3,10/3+5/3) node{\tiny$\bullet$} ;
\draw[color=red!80] (5/3,10/3+2) node{\tiny$\bullet$} ;
\draw[color=red!80] (2,10/3+4/3) node{\tiny$\bullet$} ;
\draw[color=red!80] (2,10/3+5/3) node{\tiny$\bullet$} ;
\draw[color=red!80] (2,10/3+2) node{\tiny$\bullet$} ;
\draw[color=red!80] (10/3+4/3,10/3+4/3) node{\tiny$\bullet$} ;
\draw[color=red!80] (10/3+4/3,10/3+5/3) node{\tiny$\bullet$} ;
\draw[color=red!80] (10/3+4/3,10/3+2) node{\tiny$\bullet$} ;
\draw[color=red!80] (10/3+5/3,10/3+4/3) node{\tiny$\bullet$} ;
\draw[color=red!80] (10/3+5/3,10/3+5/3) node{\tiny$\bullet$} ;
\draw[color=red!80] (10/3+5/3,10/3+2) node{\tiny$\bullet$} ;
\draw[color=red!80] (10/3+2,10/3+4/3) node{\tiny$\bullet$} ;
\draw[color=red!80] (10/3+2,10/3+5/3) node{\tiny$\bullet$} ;
\draw[color=red!80] (10/3+2,10/3+2) node{\tiny$\bullet$} ;
\draw[color=red!80] (20/3+4/3,10/3+4/3) node{\tiny$\bullet$} ;
\draw[color=red!80] (20/3+4/3,10/3+5/3) node{\tiny$\bullet$} ;
\draw[color=red!80] (20/3+4/3,10/3+2) node{\tiny$\bullet$} ;
\draw[color=red!80] (20/3+5/3,10/3+4/3) node{\tiny$\bullet$} ;
\draw[color=red!80] (20/3+5/3,10/3+5/3) node{\tiny$\bullet$} ;
\draw[color=red!80] (20/3+5/3,10/3+2) node{\tiny$\bullet$} ;
\draw[color=red!80] (20/3+2,10/3+4/3) node{\tiny$\bullet$} ;
\draw[color=red!80] (20/3+2,10/3+5/3) node{\tiny$\bullet$} ;
\draw[color=red!80] (20/3+2,10/3+2) node{\tiny$\bullet$} ;

\draw[color=red!80] (4/3,20/3+4/3) node{\tiny$\bullet$} ;
\draw[color=red!80] (4/3,20/3+5/3) node{\tiny$\bullet$} ;
\draw[color=red!80] (4/3,20/3+2) node{\tiny$\bullet$} ;
\draw[color=red!80] (5/3,20/3+4/3) node{\tiny$\bullet$} ;
\draw[color=red!80] (5/3,20/3+5/3) node{\tiny$\bullet$} ;
\draw[color=red!80] (5/3,20/3+2) node{\tiny$\bullet$} ;
\draw[color=red!80] (2,20/3+4/3) node{\tiny$\bullet$} ;
\draw[color=red!80] (2,20/3+5/3) node{\tiny$\bullet$} ;
\draw[color=red!80] (2,20/3+2) node{\tiny$\bullet$} ;
\draw[color=red!80] (10/3+4/3,20/3+4/3) node{\tiny$\bullet$} ;
\draw[color=red!80] (10/3+4/3,20/3+5/3) node{\tiny$\bullet$} ;
\draw[color=red!80] (10/3+4/3,20/3+2) node{\tiny$\bullet$} ;
\draw[color=red!80] (10/3+5/3,20/3+4/3) node{\tiny$\bullet$} ;
\draw[color=red!80] (10/3+5/3,20/3+5/3) node{\tiny$\bullet$} ;
\draw[color=red!80] (10/3+5/3,20/3+2) node{\tiny$\bullet$} ;
\draw[color=red!80] (10/3+2,20/3+4/3) node{\tiny$\bullet$} ;
\draw[color=red!80] (10/3+2,20/3+5/3) node{\tiny$\bullet$} ;
\draw[color=red!80] (10/3+2,20/3+2) node{\tiny$\bullet$} ;
\draw[color=red!80] (20/3+4/3,20/3+4/3) node{\tiny$\bullet$} ;
\draw[color=red!80] (20/3+4/3,20/3+5/3) node{\tiny$\bullet$} ;
\draw[color=red!80] (20/3+4/3,20/3+2) node{\tiny$\bullet$} ;
\draw[color=red!80] (20/3+5/3,20/3+4/3) node{\tiny$\bullet$} ;
\draw[color=red!80] (20/3+5/3,20/3+5/3) node{\tiny$\bullet$} ;
\draw[color=red!80] (20/3+5/3,20/3+2) node{\tiny$\bullet$} ;
\draw[color=red!80] (20/3+2,20/3+4/3) node{\tiny$\bullet$} ;
\draw[color=red!80] (20/3+2,20/3+5/3) node{\tiny$\bullet$} ;
\draw[color=red!80] (20/3+2,20/3+2) node{\tiny$\bullet$} ;

%% La subdivision
\draw[dashed] (10/3,-0.2) -- (10/3,10.2) ;
\draw[dashed] (20/3,-0.2) -- (20/3,10.2) ;
\draw[dashed] (-0.2,20/3) -- (10.2,20/3) ;
\draw[dashed] (-0.2,10/3) -- (10.2,10/3) ;

\draw[<->] (0,10.24) -- (10/3,10.24) ; 
\draw[<->] (10/3+4/3,10/3+2) -- (10/3+5/3,10/3+2) ;
 
\draw (2,10.24) node [above] {$(P^N/N)^{1/3}$} ; %
\draw (10/3+9/6,10/3+2) node [above] {$2d^N_m$} ; 

%% L'ouvert Omega 

{
}
\end{tikzpicture}
\caption{Second counter-example configuration. \label{SPC}}
\end{center}

\end{figure}
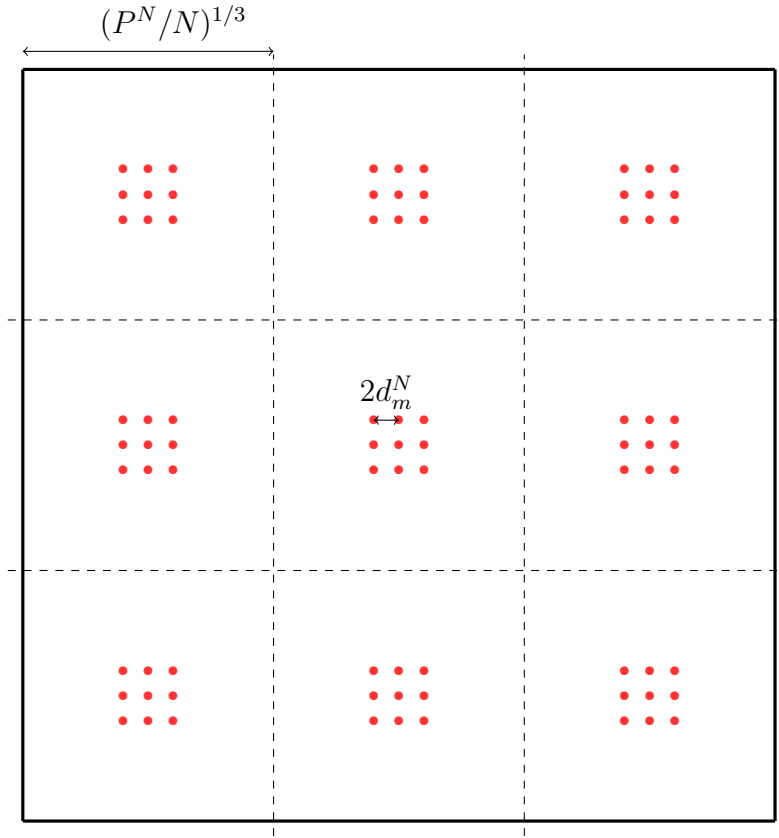

\medskip

With these conventions, we introduce $f \in L^2(\Omega)$ and are interested now in the asymptotic behavior of the unique
$u^N \in H^1_0(\Omega)$ such that there exists a pressure $p^N$ for which there holds:
 \begin{equation} \label{eq_stokesN_ccl}
\left\{
\begin{array}{rcl}
- \Delta u^N + \nabla p^N &=& f\,, \, \\
{\rm div} \, u^N &= & 0 \,,
\end{array}
\right.
\quad \text{ on $\mathcal F^{N} := \Omega \setminus \bigcup_{i=1}^N B(h_i^N,1/N)$}\,,
\end{equation}
completed with boundary conditions 
\begin{equation} \label{cab_stokesN_ccl}
\left\{
\begin{array}{rcll}
u^N &=& 0  \,, &  \text{on $\partial  B(h_i^{N},1/N)$} \,, \\
u^N &=& 0 \,, & \text{on $\partial \Omega$}\,.
\end{array}
\right.
\end{equation}

\medskip

We observe that the Stokes regime computed in \cite{Allaire} extends to this example:
\begin{proposition}
Assume that \eqref{eq_geom1}-\eqref{eq_geom2} are in force together with:
\begin{equation} \label{eq_geom3}
\lim_{N \to \infty} \dfrac{Nd_m^N}{|P^N|^{\frac 23}} = 0\,.
\end{equation}
Then, the sequence $u^N$ converges in $H^1_0(\Omega)-w$ to  the unique $\bar{u} \in H^1_0(\Omega)$
such that there exists $\bar{p} \in L^2(\Omega)$ for which:
$$
\begin{array}{rcl}
-\Delta \bar{u} + \nabla \bar{p} &=& f \\
{\rm div} \, \bar{u}&=&  0 
\end{array}
\quad \text{ in $\Omega.$}
$$
\end{proposition}
\begin{proof}
First, as $d_{m}^N >> 1/N,$ we may reproduce the arguments in Section \ref{sec_uniform} to obtain that $u^N$ is bounded in $H^1_0(\Omega).$ 
We have thus a weak-cluster point in the weak-topology. We then show that any cluster point of the sequence $u^N$ for the weak topology
of $H^1_0(\Omega)$ is the above $\bar{u}.$

To prove this latter property, we introduce $r^N := (|P^N|^{\frac 13} +1)d_m^N.$ Then, given a divergence-free $w \in C^{\infty}_c(\Omega)$
we set :
$$
\bar{w}^N = w - \left[ \sum_{k \in \Kappa^N} \chi\left(\dfrac{(x-x_k^N)}{r^N}\right)w - \mathfrak B_{x^N_k,r^N,2r^N}\left[ x \mapsto \nabla \chi\left(\dfrac{(x-x^N_k)}{r^N}\right) \cdot w(x)\right]\right]
$$
As all the holes are contained in the $B_{\infty}(x_k^N,r^N)$ for $k \in \Kappa^N$, we have that $\bar{w}^N \in H^1_0(\mathcal F^N)$ and is divergence-free. Because $u^N$ is a solution to the Stokes system in $\mathcal F^N$ 
we obtain then that:
$$
\int_{\Omega} \nabla u^N : \nabla \bar{w}^N  = \int_{\Omega} f \cdot \bar{w}^N.
$$
Let denote 
$$
\delta_{w}^N = \left[ \sum_{k \in \Kappa^N}  \chi\left(\dfrac{(x-x_k^N)}{r^N} \right)w - \mathfrak B_{x^N_k,r^N,2r^N}\left[ x \mapsto \nabla \chi\left(\dfrac{(x-x^N_k)}{r^N}\right) \cdot w(x)\right]\right]
$$ 
Remarking that the vector-fields in the sum have disjoint supports (see \eqref{eq_geom2}) and applying the properties of the Bogovskii operator of the appendix together with the fact that $\sharp \Kappa^N \lesssim N/P^N,$ we obtain:
$$
\|\delta_w^N\|^2_{H^1_0(\Omega)} \lesssim \sharp \Kappa^N r^N {\|w\|^2_{W^{1,\infty}}} \lesssim \dfrac{Nd_{m}^N}{|P^N|^{\frac 23}} {\|w\|^2_{W^{1,\infty}}}
$$
Thanks to assumption \eqref{eq_geom3}, we have  that $\delta_w^N$ converges strongly to $0$ in $H^1_0(\Omega)$ so that:
$$
\lim_{N \to \infty} \int_{\Omega} \nabla u^N : \nabla \bar{w}^N = \int_{\Omega} \nabla \bar{u} : \nabla \bar{w} \qquad \lim_{N\to \infty} \int_{\Omega} f \cdot \bar{w}^N = \int_{\Omega} f \cdot {w}\,.
$$
This ends the proof.
\end{proof}      
 
To conclude, with this example, we obtain here that  homogenizing
the Stokes problem in perforated domain does not yield what is expected from the first part of the article when \eqref{eq_geom3}
holds true {\em i.e.}:
\begin{equation} \label{eq_opt_SB}
P^N >> (Nd_{m}^N)^{\frac 32}. 
\end{equation}
On the other hand, it is clear that the configurations of this example satisfy assumption \eqref{eq_ass_concentration} with $\lambda^N = |P^N/N|^{1/3}.$ We obtain thus that we might not assume only \eqref{eq_ass_concentration} in order to prove convergence to a Stokes-Brinkman problem. A bound above for $\lambda^N$ such as \eqref{eq_ass_lambdaN} is mandatory. 

\medskip

However, in terms of $\lambda^N,$ we remark that the counter-example above shows that we may not expect convergence to the Stoke-Brinkman problem when:
$$
\lambda^N >> {N^{1/6}}{\sqrt{d_{min}^N }}.
$$ 
When $d_{min}^N$ decays like $1/N$ the bound below on the right-hand side becomes comparable to $|d_{min}^N|^{1/3}.$
Thus, the condition $\lambda^N \lesssim |d_{min}^{N}|^{1/3}$ appearing in \eqref{eq_ass_lambdaN} seems necessary.
Nevertheless, in \eqref{eq_ass_lambdaN} the condition $\lambda^N << 1/N^{1/6}$ also appears. This condition prevails
when $d_{min}^N << 1/\sqrt{N}.$ So, our counter-example does not show the optimality of \eqref{eq_ass_concentration}. This latter
restriction  $\lambda^N << 1/N^{1/6}$ comes from the third step in the proof of Proposition \ref{prop_decomposition} where we estimate the cost
of the deletion process. We found no obvious example to show that this condition is also necessary.

\appendix

\section{Auxiliary technical lemmas} \label{app_aux}

We recall here several standard lemmas that help in the above proofs.

\medskip

First, we recall the Poincar\'e-Wirtinger inequality \cite[Theorem II.5.4]{Galdi} which states that for arbitrary lipschitz domain $\mathcal F$, there holds:
$$
\left\| u -  \dfrac{1}{|\mathcal F|}\int_{\mathcal F} u(x){\rm d}x\right\|_{L^2(\mathcal F)} \leq C_{PW} \| \nabla u\|_{L^2(\mathcal F)}\,.
$$ 
We extensively use this inequality when $\mathcal F$ is an annulus. In this case, a standard scaling argument entails the following remark on the constant $C_{PW}$:
\begin{lemma}  \label{lem_PW}
Given $(x,\lambda,a) \in \mathbb R^3 \times (0,\infty) \times (0,1)$ there exists a constant $C_a$ depending
{only} on $a$ (and expecially not on $(x,\lambda)$) for which :
$$
\left\| u -  \dfrac{1}{|A(x,a\lambda,\lambda)|}\int_{A(x,a\lambda,\lambda)} u(y){\rm d}y\right\|_{L^2(A(x,a\lambda,\lambda))} \leq C_a\lambda \| \nabla u\|_{L^2(A(x,a\lambda,\lambda))}\,.
$$ 
\end{lemma}

\vskip 12pt

Second, we focus on the properties of the Bogovskii operators $\mathfrak B.$ This means we are interested in solving the divergence problem:
\begin{equation} \label{eq_divergence}
{\rm div} \, v = f\,, \quad \text{ on $\mathcal F\,,$}
\end{equation}
whose data is $f$ and unknown is $v.$ We recall the result due to M.E. Bogovskii (see \cite[Theorem III.3.1]{Galdi}):
\begin{lemma} 
Let $\mathcal F$ be a lipschitz bounded domain in  $\mathbb R^3.$
Given $f \in L^2(\mathcal F)$ such that 
$$
\int_{\mathcal F} f(x){\rm d}x = 0\,,
$$
there exists a solution $v := \mathfrak B_{\mathcal F}[f]  \in H^1_{0}(\mathcal F)$ to \eqref{eq_divergence} 
such that
$$
\|\nabla v\|_{L^2(\mathcal F)} \leq C \|f\|_{L^2(\mathcal F)}
$$
with a constant $C$ depending only on $\mathcal F.$ 
\end{lemma}
In the case of annuli, the above result yields the following lemma by a standard scaling argument:
\begin{lemma} \label{lem_div}
Let $(x,\lambda,a) \in \mathbb R^3 \times (0,\infty) \times (0,1).$ 
Given $f \in L^2(A(x,a\lambda,\lambda))$ such that 
$$
\int_{A(x,a\lambda,\lambda)} f(x){\rm d}x = 0\,,
$$
there exists a solution $v:= \mathfrak B_{x,a\lambda,\lambda}[f]  \in H^1_{0}(A(x,a\lambda,\lambda))$ to \eqref{eq_divergence} 
such that
$$
\|\nabla v\|_{L^2(A(x,a\lambda,\lambda))} \leq C_a \|f\|_{L^2(A(x,a\lambda,\lambda))},
$$
with a constant $C_a$ depending only on $a$ (and especially neither on $f$ nor on $(x,\lambda)$) \,.
\end{lemma}

\bigskip

\section{Proof of a covering lemma } \label{app_coveringlemma}
This appendix is devoted to the construction of coverings that are adapted to the empiric measures $S_N.$  
We prove the following general lemma:
\begin{lemma} \label{lem_coveringlemmaapp}
Let  $(d,\lambda) \in \mathbb N^* \times (0,\infty),$ $d \geq 2,$  and $\mu \in \mathcal M_+(\mathbb R^3)$ a positive bounded measure. 
There exists $(T_{\kappa})_{\kappa \in \mathbb Z^3}$ a covering of  $\mathbb R^3$ with disjoint cubes of width $\lambda$  such that 
denoting 
$$
\mathcal C^{\lambda}_{d} := \left\{ x \in \mathbb R^3 \text{ s.t. } \text{dist}\left(x, \bigcup_{\kappa \in \mathbb Z^3} \partial T_{\kappa}\right) < \dfrac{\lambda}{(d+1)}  \right\}
$$
there holds
\begin{equation} \label{eq_couloirapp}
\mu(\mathcal C_{d}^{\lambda}) \leq \dfrac{6}{d} \mu(\mathbb R^3).
\end{equation}
\end{lemma}
In Section \ref{sec_prfpart1}, we apply the previous lemma for arbitrary $N\in \mathbb N^*,$ with $\lambda= \lambda^N,$ $\, d = \delta-1$ and
$$
\mu := \dfrac{1}{N}\sum_{i=1}^N  (1+|v_i^N|^2) \delta_{h_i^N}\,.
$$
We obtain a covering $(T_{\kappa}^N)_{\kappa \in \mathbb Z^3}$ satisfying \eqref{eq_couloir}.  Assuming then $\lambda^N \leq [\text{dist}(\text{Supp}(w), \mathbb R^3 \setminus \Omega)/4]$ (this is possible as $\lambda^N \to 0$ ) we obtain that the subcovering $(T_{\kappa}^N)_{\kappa \in \Kappa^N}$  containing only the cubes that intersect $\text{Supp}(w)$ is made of cubes $T_{\kappa}^N$ that are included in the $\lambda^N$-neighborhood of $\text{Supp}(w).$  By direct computations, we obtain then that, for $\kappa \in \Kappa^N,$ the distance between $T_{\kappa}^N$ and $\mathbb R^3 \setminus \Omega$ 
is strictly positive so that $T_{\kappa}^N \subset \Omega.$

\begin{proof}
By a standard scaling argument, it suffices to prove the result for $\lambda=1.$
 Let  $d \geq 2.$ 
First, for arbitrary $k = (k_1,k_2,k_3) \in \mathbb Z^{3}$
we set:
$$
\tilde{T}_{k} = \left[ \dfrac{k_1}{d } , \dfrac{k_1+1}{d }\right[ \times  \left[ \dfrac{k_2}{d } , \dfrac{k_2+1}{d }\right[ \times \left[ \dfrac{k_3}{d } , \dfrac{k_3+1}{d }\right[
$$ 
These cubes with tildas and index $k$ are cubes of width $1/d.$ We call them "small cubes."
It is straightforward that $(\tilde{T}_{k})_{k\in \mathbb Z^3}$ forms a partition of $\mathbb R^3.$  For arbitrary 
$$
\kappa = (k_1,k_2,k_3) + \{0,\ldots,d-1\}^3\,, 
$$
we set then:
$$
T_{\kappa} = \bigcup_{k \in \kappa} \tilde{T}_{k} = \left[\dfrac{k_1}{d }, \dfrac{k_1}{d } +  1{}\right[
\times \left[\dfrac{k_2}{d }, \dfrac{k_2}{d } +  1{} \right[
 \times \left[\dfrac{k_3}{d }, \dfrac{k_3}{d } +  1{}\right[ .
$$
These cubes without tildas and with index $\kappa$ are cubes of width $1.$ We call them "large cubes".
We introduce then the $1/d$-neighborhood of the boundary of this large cube:
$$
[T_{\kappa}]_{d} :=   \bigcup_{k \in \partial \kappa} \tilde{T}_{k}\,.
$$
where
\begin{eqnarray*}
\partial \kappa &=& \left\{ k \in \{k_1,k_1+ d -1 \} \times \{k_2,\dots,k_2+ d -1 \} \times \{k_3,\dots,k_3+ d -1 \}\right\} \\
					&& \; 	\cup \left\{ k \in \{k_1,\dots,k_1+ d -1 \} \times \{k_2,k_2+ d -1 \} \times \{k_3,\dots,k_3+ d -1 \}\right\}\\
					&& \;	\cup \left\{ k \in \{k_1,\dots,k_1+ d -1 \} \times \{k_2,\dots,k_2+ d -1 \} \times \{k_3,k_3+ d -1 \}\right\}
\end{eqnarray*}
(which means taking the small cubes whose indices are in the boundary of $\kappa$).
We remark that we may split $[T_{\kappa}]_{d}$ into 6 subsets corresponding to the top, bottom, left, right, front and back faces
of the cube $T_{\kappa}.$ For instance, the bottom face of $[T_{\kappa}]_{d}$ reads:
$$
\bigcup_{k\in \{k_1,\dots,k_1+ d -1 \} \times \{k_2,\dots,k_2+ d -1 \} \times \{k_3\}} \tilde{T}_k\,. 
$$
For arbitrary $k^{\ell} =  \ell (1,1,1)\,,$ with $\ell \in \{0,\ldots,d-1\}$ we also denote
$$
\Kappa_{\ell}  = \Big\{ \kappa =  (k^{\ell} + \pi  + \{0,\ldots, d -1 \}^3)\,, \quad \pi \in d \mathbb Z^{3}\Big\} 
$$
We emphasize that $\Kappa_{\ell}$ is a set made of sets (corresponding to large cubes).
Any set $\Kappa_{\ell}$ corresponds to a partition of $\mathbb Z^3$ and then to a covering of $\mathbb R^3$ with disjoint large cubes.

\medskip

Given  $\ell \in \{0,\dots,d-1\}$ we consider now
$$
\mathcal C_{d}^{\ell} = \left\{ x \in \mathbb R^3 \text{ s.t. } \text{dist}\left(x, \bigcup_{\kappa \in \Kappa_{\ell}} \partial T_{\kappa}\right) < \dfrac{1}{(d+1)}  \right\}\,.
$$
We remark that, for fixed $\ell$ there holds:
$$
\mathcal C_{d}^{\ell} \subset \bigcup_{\kappa \in \Kappa_{\ell}} [T_{\kappa}]_{d}.
$$
We denote $\partial \Kappa_{\ell}$ the set of indices  $k$ such that $\tilde{T}_{k}$ contributes to this $1/d$-neigborhood, {\em i.e.},
$\partial \Kappa_{\ell} =  \bigcup \left\{ \partial \kappa,\; \kappa \in \Kappa_{\ell}\right\}.$ We have thus:
$$
\mathcal C_{d}^{\ell} \subset \bigcup_{k \in \partial \Kappa_{\ell}} \tilde{T}_k\,.
$$
We can decompose this union of small cubes by regrouping together the cubes that belong to left / right / top / bottom / front / back faces 
of large cubes. For instance, the indices $k$ of small cubes belonging to  bottom faces of large cubes satisfy
$$
k \in   \mathbb Z^2 \times \{\ell  + d\mathbb Z\}\,.
$$
For two different $\ell$ and $\ell'$ in $\{0,\ldots,d-1\}$ the same
index $k$ cannot belong to the bottom faces of two different cubes in the coverings $\Kappa_{\ell}$ and 
$\Kappa_{\ell'}$ of $\mathbb R^3.$ We have the same properties for top / right / left / front / back faces.
Consequently, in the family of coverings $(\Kappa_{\ell})_{\ell \in \{0,\ldots,d-1\}}$ one small cube $\tilde{T}_{k}$
belongs at most once to a top / bottom / right / left / front / back face of a large cube so that: 
\begin{equation} \label{eq_fmtlremark}
\text{ any $k \in \mathbb Z^3$ belongs to at most $6$ different $\partial \Kappa_{\ell}$}\,.
\end{equation}
 
\medskip

Let now introduce the measure $\mu.$ For any $k \in \mathbb Z^3,$ we denote:
$$
\tilde{\mu}_{k} = \mu(\tilde{T}_{k}),
$$
and we consider the sum:
$$
Rem := \sum_{\ell \in \{0,\ldots,d-1\}} \mu(\mathcal C_{d}^{\ell}).
$$
With the previous definitions, we have:
$$
Rem \leq \sum_{\ell \in \{0,\ldots,d-1\}} \sum_{k \in \partial \Kappa_{\ell}} \tilde{\mu}_k.
$$
Because of \eqref{eq_fmtlremark}, we have then that any $k \in \mathbb Z^3$ appears at most $6$ times in this sum. 
Consequently:
$$
Rem  \leq 6 \sum_{k \in \mathbb Z^3} \tilde{\mu}_k \leq 6 \mu(\mathbb R^3)\,. 
$$
The measure $\mu$ being positive and finite, this implies that one of the terms in the sum defining $Rem$
is less than $Rem/d$. In other words,  there exists at least one 
$\ell^0 \in  \{0,\ldots,d-1\}$ such that:
$$
\mu(\mathcal C_{d}^{\ell^0}) \leq  \dfrac{6}{d}\mu(\mathbb R^3)\,.
$$
The covering $(T_{\kappa})_{\kappa \in \Kappa_{\ell^0}}$ is then made of disjoint cubes 
of width $1$  satisfying \eqref{eq_couloirapp}. 
We have obtained the required covering of $\mathbb R^3.$
\end{proof}

\medskip

{\bf Acknowledgement.} The author  would like to thank Laurent Desvillettes, Ayman Moussa, Franck Sueur, Laure Saint-Raymond and Mark Wilkinson for many stimulating discussions on the topic.
The author is partially supported by the ANR projects ANR-13-BS01-0003-01 and ANR-15-CE40-0010.

\medskip

{\bf Conflict of interest.} The author declares that he has no conflict of interest.

\end{document}